\documentclass[leqno,12pt]{amsart}
\usepackage{epsfig,color,amssymb,mathtools}
\usepackage{graphicx}

\newtheorem{Thm}[equation]{Theorem}
\newtheorem{Cor}[equation]{Corollary}
\newtheorem{Lem}[equation]{Lemma}

\newtheorem{Con}[equation]{Conjecture}

\theoremstyle{definition}
\newtheorem{Def}[equation]{Definition}

\theoremstyle{remark}

\newtheorem{Rem}[equation]{Remark}

\numberwithin{equation}{section}
\makeatletter
\renewcommand{\c@figure}{\c@equation}
\makeatother

\newcommand{\bord}{\partial}

\newcommand{\dem}{{\emph{Proof.} \  }}
\newcommand{\eop}[1]{{\flushright\hfill\fbox{#1}}}

\begin{document} 

\title[shape recognition of planar domains]{Approximation of conformal mappings and novel applications to shape recognition of planar domains}

\author{Sa'ar Hersonsky}

\address{Department of Mathematics\\ 
University of Georgia\\ 
Athens, GA 30602}

\urladdr{http://saarhersonsky.wix.com/saar-hersonsky}
\email{saarh@uga.edu}
\date{}
\thanks{This work was partially supported by a grant from the Simons Foundation (\#319163 to Saar Hersonsky). }

\begin{abstract} 
Our goal is to provide a novel method of representing 2D shapes, where each shape will be assigned a unique fingerprint - a computable approximation to the 
conformal map of the given shape to a canonical shape in 2D or 3D space (see page 22 for a few examples). In this paper, we make the first significant step in this program where we address  the case of simply, and doubly-connected planar domains. We prove uniform convergence of our approximation scheme to the appropriate conformal mapping.


To this end, we affirm a conjecture raised by Ken Stephenson in the 90's  which predicts that  the Riemann mapping can be  approximated by a sequence of electrical  networks. In fact, we first treat a more general case.
Consider a planar annulus, i.e., a bounded, $2$-connected, Jordan domain,  endowed with a sequence of triangulations exhausting it. We  construct a corresponding sequence of
maps which converge uniformly on compact subsets of the domain, to a conformal homeomorphism onto the interior of a Euclidean annulus bounded by two concentric circles. The resolution of Stephenson's Conjecture then follows by a limiting argument.

With more complex topology of the given shape, i.e, when it has higher genus,  we will  use methods invented by Arabnia \cite{Ara} and Wani-Arabnia \cite{WanAra}. First, to divide the domain into subdomains and thereafter to make the scheme presented in this paper suitable for parallel processing. We will then be able to compare our results for those appearing, for instance, in the work of Arabnia-Oliver \cite{AraOli} that provides algorithms for the translation and scaling of complicated digitalized images.


\end{abstract}
\maketitle
\section{Introduction}
\label{se:Intro}  

\subsection{Riemann's Mapping Theorem and Thurston's disk packing scheme}
\label{se:perspective}

The Riemann Mapping Theorem asserts that any simply connected planar domain which is not the whole plane, 
can be mapped bi-holomorphically onto the open unit disk.  That is, the domains are {\it conformally equivalent}.  After a suitable normalization, this mapping is called the Riemann mapping and  
it is  desirable to have a concrete approximation of it.  In \cite{RoSu}, Rodin and Sullivan proved Thurston's celebrated conjecture \cite{Th} asserting that a scheme based on the Koebe-Andreev-Thurston disk packing theorem    (cf. \cite{An1,An2,Ko1,Th1}) 
converges to the Riemann mapping.   

\smallskip

In order to formulate Thurston's conjecture, which inspired Stephenson's conjecture, we need  to recall  a few definitions.  Let $P$ be a {\it disk packing} in the complex plane $\mathbb{C}$.  An {\it interstice} is a connected component of the complement of $P$, and  one whose closure intersects only three disks in $P$ is called a {\it triangular interstice}. We will let  $\mbox{\rm supp}(P)$ denote the union of the disks in $P$ and all its  bounded interstices, and we will assume that it is  simply-connected. The disks of $P$ that intersect the boundary of its support are called {\it boundary disks}. Two finite disk packings, $P$ and $\tilde P$ in $\mathbb{C}$, will be called {\it isomorphic}, if there exists an orientation preserving homeomorphism $\phi:\mbox{\rm supp}(P)\rightarrow \mbox{\rm supp}(\tilde P)$ such that $\phi(P)= \tilde P$. It is clear that such an isomorphism induces a bijection between the disks of $P$ and the disks of $\tilde P$. 

\smallskip

Let $\Omega \subsetneq \mathbb{C}$ be a bounded, simply connected domain, and let $p_0$ be an interior point in it. For each positive integer $n$, let $P^{n}$ be a disk packing in $\Omega$ in which all bounded interstices are triangular. Assume that there is a sequence of positive numbers $\delta_n$ which converges to zero, such that: 
i)
 the radius of every disk in $P^n$ is smaller than $\delta_n$, and 
ii) every boundary disk in $P^n$ is  at most the distance of $\delta_n$ from $\partial \Omega$. Finally,  let $P_0^n$ be a selected disk in $P^n$ which is closest to $p_0$ or contains it.

\smallskip

The Disk Packing Theorem (Koebe-Andreev-Thurston) implies  that there exists an isomorphic packing ${\tilde P}^n$ in the closed unit disk $\bar{\mathbb D}$ with all of its boundary disks  tangent to the unit circle ${\mathbb S}^{1}$.   Furthermore,  if the given graph is isomorphic to the $1$-skeleton of a triangulation of the Riemann sphere, then the packing is unique up to applying a  M\"obius transformation. Let 
\begin{equation}
\label{eq:maps0}
f_n: \mbox{\rm supp}(P^n)\rightarrow \mbox{\rm supp}(\tilde P^n)
\end{equation}
 be an isomorphism of $P^n$ and ${\tilde P^n}$. Furthermore, normalize ${\tilde P^n}$ by a sequence of M\"{o}bius transformations preserving $U$ so that ${\tilde P}_0^n$, the disk corresponding to  $P_0^n$, is centered at the origin. Thurston conjectured that if the packings $P^n$ are chosen to be sub packings of scaled copies of the infinite hexagonal disk packing of $\mathbb{C}$,  then the sequence of piecewise affine maps (i.e., simplicial)   $f_n$ converges uniformly on compact subsets of $\Omega$ to the Riemann mapping from $\Omega$ to ${\mathbb D}$.

\smallskip

Rodin and Sullivan  \cite{RoSu} proved Thurston's Conjecture by first showing that the maps $f_n$ are $K$-quasiconformal, for some fixed $K$. Hence, there exists a subsequence which will converge to a limit function $f$ which must also be $K$-quasiconformal. 
Rodin and Sullivan further showed that $f$ must be $1$-quasiconformal, and therefore, $f$ is in fact conformal. He and Schramm  \cite[Theorem 1.1]{HeSch} developed profound  techniques which avoid the machinery of quasiconformal mapping that is heavily used in Rodin-Sullivan's proof. Up to date, their theorem and advances \cite{HeSch1} in the simply connected case, is the most advanced. See also their related work on Koebe's Conjecture in \cite{HeSch2}.  

\smallskip

Chow and Luo \cite{ChLu} discovered applications of disk packing to the study of discrete  Ricci flow on surfaces; see also the work of Glickenstein \cite{Gl} for related study. 
There are also applications of circle packings to algorithmic computer vision and  computational conformal geometry due to Gu, Luo and Yau,  Gu, Zeng, Zhang, Luo and Yau,  and Sass, Stephenson and Brock  (cf.  \cite{GuZeLuYa,GuLuYa} \cite{Zh} and \cite{SaStBr} as examples and further advances). More recently, taking a complementary approach to the
one in this paper, Gu, Luo, Sun, and Wu \cite{GuLuSuWu} have developed powerful tools establishing several important results   concerning discrete uniformization of polyhedral surfaces.




\subsection{Electrical networks and Stephenson's conjecture}
\label{se:question}


In his attempts to prove {\it uniformization},  Riemann suggested considering a planar annulus as made of a uniform conducting metal plate. When one applies voltage to the plate, keeping one boundary component at a fixed voltage $k$ and the other at the  voltage $0$, {\it electrical current}\   will flow through the annulus.  The  {\it equipotential} sets  form a family of disjoint, simple closed curves foliating the annulus and separating the boundary curves. The \emph{current} flow sets consist of simple disjoint arcs connecting the boundary components,  and they also foliate the annulus. Together, the two families provide  curved ``rectangular" coordinates on the annulus that can be used to turn it into  a right circular cylinder, or into a (conformally equivalent) circular concentric annulus. 

\smallskip

An {\it electrical circuit} or {\it network}  is a collection of nodes and connecting wires. For instance, a disk packing of a fixed planar domain induces such a network where 
each center of a disk  corresponds to a node and a wire connects each pair of nodes corresponding to tangent boundaries. 
It is therefore reasonable to conjecture that if the domain is made of thin conducting material then its electrical behavior can be approximated by a sequence of networks that approximates its {\it shape}. 

\smallskip

Stephenson's Conjecture from the 90's (see page 63 and Definition 6.5.1 in \cite{Ste2})  is concerned with constructing such an approximation:
\begin{Con}[Stephenson \mbox{\cite{Ste2}}]
\label{con:step}
Given a sequence of networks approximating a simply-connected, bounded, Jordan  domain arising, for instance, from a sequence of disk packing, 
choose conductance constants along the edges (for each network) according to Equation~(\ref{eq:condu}).
Then the sequence of discrete potentials and currents will converge to the ones induced by the Riemann mapping.
\end{Con}

We have phrased this conjecture  in the more recent formulation of (\ref{eq:condareequal}) (see Section~\ref{se:stephcondu} for the details). In fact, a similar conjecture can be formulated for any domain that can be approximated (in a sense that we will make precise in Section~\ref{se:anncont}) by a sequence of quasi-uniform triangulations (see Definition~\ref{de:regular}) that exhaust the given domain. 

\smallskip

In Theorem~\ref{th:annulus2}, we will formalize and affirm Stephenson's conjecture in the case of an annulus by methods that are different from the ones used in his paper or those mentioned in Section~\ref{se:perspective}. In particular, we will show that there exists  a large class of networks for which the conjecture holds. We will also affirm this conjecture in its original form, i.e., for simply-connected domains in the complex plane.

\subsection{The themes of this paper}
\label{se:plan}

There is a classical and  elaborate theory of conformal uniformization for domains in the Riemann sphere that are bounded by Jordan curves. Let $\Omega$ be such a domain which is also finitely connected.  Koebe proved  \cite{Ko} that $\Omega$ is conformally homeomorphic to some domain $\Omega^{*}$ whose boundary components are circles. Such a domain is called a {\it circle domain}.  Furthermore, $\Omega^{*}$ is unique up to M\"obius transformations.


\smallskip
Discrete uniformization schemes have traditionally been the first step in constructing a sequence of approximations to a conformal map from the given domain (more on this in Section~\ref{se:perspective}).
There is much interest and effort by, for example, Cannon, Floyd and Parry, to provide sufficient combinatorial conditions under which, discrete schemes based on the {\it discrete extremal length} method, will converge to a conformal map in the cases of triangulated  annulus or a quadrilateral; see for instance \cite{Ca} for the starting point and  \cite{CaFlPa3} for their most recent work. However,  Schramm showed  \cite[page 117]{Sch} that if one attempts to use the combinatorics of 
the hexagonal lattice alone, square tilings (as constructed by Schramm's method) cannot be used as discrete approximations for the Riemann mapping. 

\smallskip
In a different vein, 
of much current and recent interest is 
the universality of the critical Ising and other lattice models where discrete complex analysis on graphs played a crucial role (see for instance \cite{ChSm, DCSm}).

\smallskip
In this paper, stemming from our work in \cite{Her1,Her2,Her3,HAT},   we will prove that a certain discrete scheme yields convergence of the mappings described below to a canonical 
{\it conformal} mapping from a given polygonal, planar, annulus, onto the interior of a Euclidean annulus bounded by two concentric circles.   

\smallskip

Specifically, the underlying idea of this paper is rooted in a foundational feature of two dimensional conformal maps. If $f:\mathbb{D}\rightarrow \mathbb{C}$ is a conformal map, then the Cauchy-Riemann equations imply that $\Re(f)$ and $\Im(f)$ are harmonic functions, and that $\Im(f)$ is the harmonic conjugate of $\Re(f)$. For instance,  when $(r,\theta)$ are  polar coordinates in the plane, we have  that 
$u(r,\theta)= \log{r}$ and $v(r,\theta)=\theta$ (when $\theta$ is single valued) are harmonic functions, and $v(r,\theta)$ is the harmonic conjugate of $u(r,\theta)$.  Indeed,
in this paper, we will work with the pair $(g, \bar g^{*})$ which are  
 {\it combinatorial functions} defined on the triangulation and 
its Voronoi dual (to be explained later).



\smallskip

In Theorem~\ref{th:annulus2}, we will show that under certain geometric restrictions on the sequence of triangulations, where each triangulation is  endowed with the conductance constants defined according to Equation~(\ref{eq:condu}), the sequence of combinatorially   defined functions 
$$
\phi_n=\exp\big(\frac{2\pi}{\mbox{\rm period}(\bar g_n^{\ast})}\big( g_n  + i \bar g_n^{*})\big),
$$
 will converge uniformly on compact subsets of a given annulus, to the conformal uniformizing map of the annulus whose form is well understood (see for instance \cite[Section 7]{Cou} or \cite[Theorem 4.3]{Tay}).  

\smallskip
To this end, 
we will employ $L_{\infty}$ convergence results from the theory of the \emph{finite element method}, techniques from discrete potential theory, and classical results form the theory of functions of one complex variable concerning compactness of sequences of holomorphic mappings,  and partial differential equations.  In order to put some of the needed advances over previous work in context,  let us  briefly recall an inspiring  work by Dubejko \cite{Dub}. Let $w$ denote the solution of the Dirichlet problem $\Delta w=f$ for  $x\in\Omega$, and $w=\phi$ for  $x\in \partial\Omega$, where $\Omega$ is a simply-connected, bounded,  Jordan domain with $C^{2}$ boundary, where $f\neq 0\in L^{2}(\Omega)$ and $\phi \in C^{0}(\Omega)$. By applying techniques from the finite volume method, Dubejko proved that  $w$   
can be approximated (in various norms) by a sequence of solutions of discrete Dirichlet boundary value problems. These solutions  belong to a certain \emph{Sobolev  space} and are constructed via a sequence of triangulations (with special properties)  that gets finer while exhausting $\Omega$ from the inside.  Dubejko's work, which utilized  Stephenson's conductance constants in the setting of the \emph{finite volume method} (see \cite{EyGaHe1}),  is not sufficient for constructing approximations of conformal maps from Jordan domains. In fact, already in the simply connected case  his techniques  are not sufficient. This is due to the following reasons:
His methods can be applied only under the assumption that the boundary of $\Omega$ is $C^2$;
second, Dubejko did not address the problem of defining a combinatorial analogue of the harmonic conjugate;  finally, Dubejko applied the  Riemann's mapping theorem in his proof. 

\smallskip

In order to overcome some of these  issues, we will employ a foundational  result from the theory of the finite element method \cite[Theorem 4.1]{ScWa}. This result will provide the  $L_{\infty}$ convergence of the (normalized) $g_n$'s, which are \emph{different} from the ones used by Dubejko, to the real part of the uniformizing map of an annulus with continuous boundary. Once this convergence result is applied, one novel part of this work  is  introducing a combinatorial analouge of the harmonic conjugate function and proving its convergence to its analytical counterpart. 

\subsection{Organization of the paper}
\label{se:plan1}
In Section~\ref{se:stephcondu}, we start by recalling the definition of the conductance constants suggested by Stephenson in his conjecture (Conjecture~\ref{con:step}). We then express these in the way they are going to be utilized in Theorem~\ref{th:annulus2}, the main theorem of this paper, which proves that a certain discrete scheme converges to a uniformizing map of a planar annulus.

\smallskip

In Section~\ref{se:makeconf}, we present three novel  definitions.  First, we define the class of {\it discrete asymptotic harmonic functions}.  Intuitively, a  function in this class is {\it almost} harmonic on a scale determined by the mesh of the triangulation. Since our discrete approximations of  $\Re{(f)}$, the \underline{smooth} uniformizing  map of an annulus, is {\it not} discrete harmonic,  introducing this class of functions is essential to the approximation process described in the main theorem of this paper.  Second, the {\it flux following path} is contained in the one skeleton of a given triangulation; it is used to determine the amount of {\it discrete flux} of a function which ``crosses" a path in the one skeleton of the Voronoi cells of the given triangulation. Finally, 
if $g$ is a discrete harmonic or a discrete asymptotic harmonic function, 
by summing the discrete flux along such paths, we are able to define a  {\it conjugate function}  $\bar g^{*}$ of $g$ and thereafter to prove its convergence. 
 
\smallskip
Section~\ref{se:makeconf11} is devoted to the approximation of a uniformizing map of a planar annuli with a continuous Jordan  boundary. We  first study the case of a polygonal annulus. 
In Theorem~\ref{th:annulus2}, we prove the uniform convergence of our proposed discrete scheme  on compact subsets of the interior of the given annulus, to a conformal homeomorphism.
We are then 
 concerned with the approximation of the uniformization of an  annulus with continuous  Jordan boundary. 
 Corollary~\ref{th:anncont}  demonstrates that Theorem~\ref{th:annulus2} coupled with 
a generalization of a compactness theorem due to  Koebe and a diagonalization process,
 allow the weakening of the boundary regularity assumption  of Theorem~\ref{th:annulus2}  from polygonal to continuous. 

\smallskip

Section~\ref{se:disk} is devoted to the proof of Theorem~\ref{th:disk}, where we provide an approximation of the uniformization of a bounded, simply-connected Jordan domain, the setting in which Stephenson's conjecture (Conjecture~\ref{con:step}) was first stated. The idea is to present the {\it punctured} domain  as an increasing sequence of annuli. Thus, one can apply Theorem~\ref{th:annulus2} to each annulus in the sequence. 
The existence of a converging subsequence of the maps obtained in each step to a bounded, conformal, univalent map is then proved (following the same rationale as  in Corollary~\ref{th:anncont}), and we can therefore restrict attention to the case that the boundary of the domain is polygonal.  Finally,  the  Riemann's removable singularity theorem is used to show that the sequence of the above conformal maps is bounded,  hence, can be extended over the puncture. 

\smallskip

With the aim of  making this paper self-contained, it contains an Appendix.  In Appendix~\ref{se:notation1} and in Appendix~\ref{se:ENFT}, we collect a few important notations, definitions and theorems from the finite element method that are applied in this paper. The reader who is familiar with this method, can skip these sections. However, Theorem~\ref{Th:FVMEtimate}, which is quoted from \cite[Theorem 4.1]{ScWa} is essential for the $L_{\infty}$ convergence analysis  results of this paper. In Appendix~\ref{se:ENDF}, we  describe  the relation between Stephenson's conductance constants and the theory of the finite volume element method.

\subsection*{Acknowledgement}  \small{We are indebted to Gilles Courtois  for carefully reading this paper, a great deal of help in clarifying several proofs  and in making the presentation lucid.  Ridgway Scott, Thierry Gallouet, and Al Schwatz are heartily thanked for graciously sharing with the author their insights regarding the subtle analysis involved in numerical methods for convex and non-convex planar domains. We are grateful to Benson Farb, Rich Schwartz, Ted Shifrin, and  Robert Varley for their patient listening, and 
valuable discussions during the preparation of this paper. Our gratitude to Eric Perkerson for essential help in implementing   the algorithm in Theorem~\ref{th:annulus2} (see for instance the illustrations on page 22). 


\section{Electrical networks  induced by disk packings and Stephenson's conductances}
\label{se:stephcondu}

Let us recall  a few definitions and some notation from \cite{Dub0, Dub,Ste2} and  \cite{Ste3}  in order to express the conductance constants suggested by  Stephenson. Let $P$ be a euclidean disk packing of a domain $\Omega$ for a complex $K$, i.e.,  the contact graph of $P$ is isomorphic to $K$. For an interior edge $(u,v)\in K$, consider the tangent circles, $c_v$ and $c_u$, as depicted in the figure below. Let $c_x,c_y$ be their common neighboring circles.

\smallskip

\begin{figure}[htbp]
\begin{center}
 \scalebox{.55}{ \input{circles.pstex_t}}
 \caption{Constructing an edge conductance in a circle packing.}
\label{figure:quad1}
\end{center}
\end{figure}

The {\it radical center}, $w_x$,  of the triple $\{c_v,c_u,c_x\}$ of circles will denote the center of the circle that is orthogonal to $c_v,c_u$ and $c_x$ and let $w_y$ be the radical center of the triple $\{c_v,c_u,c_y\}$. Let $z_u,z_v$ be the centers of $c_u,c_v$, respectively. Finally, for a vertex $v$, let $R_v$ denote the radius of the circle $c_v$. The {\it sum of the angles} at $v\in P$ is obtained by adding the angles formed by the edges of the contact graph  of $P$ emanating from  $z_v$.

\smallskip

Stephenson's {\it conductance} of an edge
is defined by (see also Definition~\ref{de:cond1} and Equation~(\ref{eq:overatriangle})):
 \begin{equation}
\label{eq:condu}
c(e)=c(u,v)=\frac{|w_x-w_y|}{|z_u-z_v|}.
\end{equation}

It is illuminating to give a probabilistic interpretation to this quantity.  Stephenson's main idea was to chase angle changes at the centers of the circles, as radii change while maintaining (new) disk packing. Given a euclidean circle packing,  the effect of a small {\it increase} in the radius of one of the circles, say $R_v$, is that the sum of the angles at $v$ decreases, while the angle sums  at the neighboring vertices $\{v_1,v_2,\ldots, v_k\}$ {\it increase}. Some of the angle ``distributed" by $v$ arrives at $v_j$ and must be passed along in order to keep a packing at $v_j$. Hence, $R_{v_j}$ has to be adjusted and we need to keep track of the angle changes of its neighbors, and so forth.

\smallskip

In Euclidean geometry, the angles of any triangle add up to $\pi$, so angles in this process will never get lost. In other words, the total angle {\it leaving} one vertex must be divided into portions and then distributed as angles {\it arriving} to its neighbors. This movement can be expressed as a {\it Markov process}, where the transition probability from $v$ to $v_j$, is the proportion  of a {\it decrease} in the sum of the angles at $v$ that becomes an {\it increase} in the sum of the angles at $v_j$. In this Markov process, the random walkers are the quantities of {\it angles} moving from one vertex to another. Thus, for a specific neighbor $u=v_j$, the amount of angle arriving at 
$\psi_u$ is given by     
$\frac{d\psi_u}{dR_v}$. It turns out that the transition probability from $v$ to $u$ as described above is given by 
\begin{equation}
\label{eq:Markov}
\bar\rho(v,u)=\dfrac{\frac{d\psi_u}{dR_v}}{ \sum_{j=1}^{k} \frac{d\psi_{v_j}}{dR_{v_j}}}.
\end{equation}
Also, for a vertex $v\in K$, we  let 
\begin{equation}
\label{eq:condu1}
\rho(v,u)= \dfrac{ c(v,u) }{ \sum_{u\sim v} c(v,u) } .
\end{equation}

It is remarkable that in 2005  (see \cite[Section 18.5]{Ste0})
Stephenson showed  that equality holds between these two Markov transitions, that is,
\begin{equation}
\label{eq:condareequal}
\rho(v,u)=\bar\rho(v,u),\  u\sim v.
\end{equation}

 \section{Smooth harmonic conjugate functions and their combinatorial counterparts.}
\label{se:makeconf}
This  section entails several key definitions and constructions. 
In the first subsection, we collect a few classical PDE existence results that go back to Poincar\'e and Lesbegue. In the second subsection,  we will assume that ${\mathcal A}$ is a fixed, planar, polygonal annulus endowed with a triangulation ${\mathcal T}$. We will write $\partial {\mathcal A}= E_1\cup E_2$ where  $E_1$ denotes  the outer boundary component.

\smallskip

After recalling the definitions of the 
{\it combinatorial laplacian} and the {\it normal derivative}, we will turn to define the class of discrete, asymptotically harmonic functions  (this class includes discrete harmonic functions). The main goal of this section is 
to define a conjugate function to any function in this class (see Definition~\ref{de:asyharmonic}).
One interesting feature of the conjugate function is that, in general, and unlike the smooth category, it is not harmonic.

\subsection{Strong solutions of the Laplace equation and smooth harmonic conjugate functions}
\label{se:back}
The solutions of the Laplace equation, {\it  harmonic functions}, have a foundational role in various areas of mathematics. In this paper, we will 
apply known connections between harmonic functions and conformal maps defined on $\Omega$. Furthermore, we will later on use an approximation scheme of the solution in our construction of a combinatorial analogue of the harmonic conjugate function. 
Let  $\Omega$ be a bounded, planar domain and assume that $u\in C^{2}(\Omega)\cap C(\bar\Omega)$ is  the {\it strong solution} of the Dirichlet boundary value problem for the Laplace equation with non-homogeneous boundary conditions 
\begin{equation}
\label{de:strong0}
\begin{cases}
      & \triangle u=0\     \text{in}\  \Omega, \\
      & u= h\  \text{on}\ \partial\Omega,
\end{cases}
\end{equation}
where $h\in C(\partial \Omega)$, or more generally is the trace of $\tilde{h}\in H^{1}(\Omega)$.


\smallskip

The study of the existence of strong solutions of Dirichlet  boundary value type problems has an impressive history. Poincar\'e introduced the notion of \emph{barriers}, and their importance was further recognized later  by Lebesgue.
A  function $w\in C^{0}(\Omega)$ is called {\it super-harmonic}  in $\Omega$, if for any closed region $\Omega'\subset\Omega$, and any harmonic function $u$ in the interior of $\Omega'$, whenever the inequality 
\begin{equation}
\label{eq;sh}
w\geq u 
\end{equation}
holds on the boundary of $\Omega'$, it also holds in the interior of $\Omega'$.

\smallskip

Let $\xi$ be a point in $\partial \Omega$, then a $C^{0}(\bar \Omega)$ function $w=w_\xi$ is called a
{\it barrier} at $\xi$ relative to $\Omega$. If
 $w$ is super-harmonic in $\Omega$, 
   it approaches $0$ at $\xi$, and outside of any sphere about $\xi$,  it has a positive lower bound in $\Omega$.  Two profound consequences of the existence of a barrier are the following.
\begin{Thm}[\mbox{\cite[Theorem III, page 327]{Kel}}]
\label{th:exist}
A necessary and sufficient condition that the Dirichlet problem for $\Omega$ is solvable for  arbitrary assigned continuous boundary values,  is that a barrier for $\Omega$ exists at every point in $\partial \Omega$.
\end{Thm}

It is therefore important to understand which domains in the Euclidean plane  satisfy the hypothesis of Theorem~\ref{th:exist}. Indeed,  general sufficient conditions  can be described in terms of local 
properties of the boundary (see for instance \cite[Proposition 5.13]{Tay}).

\begin{Thm}[Lebesgue]
\label{th:existbarrier}
The Dirichlet boundary value problem  (\ref{de:strong0})  is solvable  for arbitrary assigned continuous boundary  values
if every component of the complement of the domain consists of more than a single point. 
\end{Thm}

For the applications of this paper,  the following corollary is essential.
\begin{Cor}
\label{co:Jordan domains}
Let $\Omega$ be a Jordan domain, then the Dirichlet boundary problem (\ref{de:strong0}) is solvable in $\Omega$  for arbitrary continuous boundary values.
\end{Cor}
The (strong)  maximum principle (see for instance \cite{Chi}) implies that a strong solution is unique. Therefore, in the special case studied in the this paper, where $\Omega={\mathcal A}$ is a  planar annulus (with polygonal or even continuous boundary),  we make the following:

\begin{Def}
\label{de:solutiononann} We call  $u\in C^{2}({\mathcal A})\cap C^{0}(\bar{\mathcal A})$  the strong solution of the Dirichlet boundary value problem of the Laplace equation, if 

\begin{equation}
\label{de:strong}
\begin{cases}
      & \triangle u=0\     \text{in}\ {\mathcal A}, \\
      &\  u= 1\ \text{on}\ E_1,\  \text{and}\ u=0\  \text{on}\ E_2.
\end{cases}
\end{equation}
\end{Def}

\smallskip


We end this subsection by recalling the following definition which is valid for any harmonic function. 
\begin{Def}[A smooth harmonic conjugate (\mbox{see for instance \cite[Chapter 1.9]{Neh}})]
\label{de:smoothconju}
Let $(x_0,y_0)$ be a point in ${\mathcal A}$,  and let $(x,y)$ in ${\mathcal A}$ be an arbitrary point. Let $\gamma$ be a simple, piecewise-smooth 
curve joining $(x_0,y_0)$ to $(x,y)$ in ${\mathcal A}$. Let $\beta$ be any simple, closed,  counter-clockwise oriented,  piecewise smooth curve in ${\mathcal A}$ whose winding number is equal to 1. Furthermore,  let $s$ denote the arc-length parameter of these curves, and let $\hat n$ denote a unit normal pointing to the right of the tangents to these curves.

\smallskip

A  (multivalued) harmonic conjugate of $u$ is defined by 
\begin{equation}
u^{*}(x,y)= u^{*}(x_0,y_0) +\int_{\gamma}\frac{\partial u}{\partial \hat n} ds,\  
\end{equation}
where $u^{*}(x_0,y_0)$ is some arbitrary fixed real constant, and the period of $u^*$ is defined by 
\begin{equation}
\label{eq:periodsmooth}
\text{period}(u^{*})=\int_\beta \frac{\partial u}{\partial \hat n} ds.
\end{equation}
\end{Def}


\begin{Rem}
\label{re:pathind}
It is well known that a smooth harmonic conjugate $u^{*}$  is defined up to a constant, i.e.,  an assigned value at a point in the annulus. Furthermore, the function values at any point differ by integral multiples of its period, i.e., $u^*$  is multivalued. 
\end{Rem}

\subsection{Discrete harmonic and asymptotically harmonic functions, and their conjugates}
\label{se:asyhar}

We now turn to defining a combinatorial function analogous to $u^{*}$.   We will start with some notation and definitions from  the subject of discrete harmonic analysis  that will be used throughout the rest of this paper (see for instance \cite{BeCaEn}  or \cite[Section 1.1]{Her3}). Let $\Gamma=(V,E,c)$  be a planar {\it finite network}; that is,  
  a planar, simple, and
finite connected graph with vertex set $V$ and edge set $E$, where each edge $(x,y)\in E$ is
assigned a {\it conductance} $c(x,y)=c(y,x)>0$. Let ${\mathcal P}({V})$
denote the set of non-negative functions on $V$. Given $F\subset V$, we denote by $F^{c}$ its complement in
$V$.  Set
${\mathcal P}(F)=\{u\in {\mathcal P}(V):S(u)\subset F\}$, where $S(u)=\{ x \in V: u(x)\neq 0 \}$.  The set  $\delta F=\{x\in F^{c}: (x,y)\in E\ {\mbox
{\rm for some}}\ y\in F \}$ is called the {\it vertex boundary} of
$F$. Let ${\bar F}=F\cup \delta F$, and let $\bar E=\{(x,y)\in
E :x\in F\}$.
Let
${\bar \Gamma}(F)=({\bar F},{\bar E},{\bar c})$ be the network
such that ${\bar c}$ is the restriction of $c$ to ${\bar E}$. 
We write $x\sim y$ if $(x,y)\in \bar E$, $y$ is  called a neighbor of $x$, and we let $N_x$ denote the cardinality of the set of neighbors of $x$.  The following operators are discrete analogues of
classical notions in continuous potential theory (see for instance \cite{Fu} and  \cite{ChGrYa}).

\begin{Def}
 \label{de:lapandnorm}  
 Let $u\in {\mathcal P}({\bar  F})$. 
 Then 
 for $x\in F$, the function 
 \begin{equation}
 \label{eq:lap}
 \Delta u(x)=\sum_{y\sim x}c(x,y)\left( u(x)-u(y) \right )
 \end{equation} is called
  the \emph{Laplacian} of $u$ at $x$.  For $x\in \delta(F)$, let $\{y_1,y_2,\ldots,y_m\}\in F$ be its neighbors.
  
\smallskip  
  
The \emph{normal derivative} of $u$ at a point
$x\in \delta F$ with respect to a set $F$ is defined by 
\begin{equation}
\label{eq:nor}
\frac{\bord u}{\bord n}(F)(x)= \sum_{y\sim x,\
y\in F}c(x,y)  (u(x)-u(y)).
\end{equation} 
Finally,  $u\in {\mathcal P}({\bar F})$ is called \emph{discrete harmonic} in $F\subset V$ if
$\Delta u(x)=0,$ for all $x\in F$.
\end{Def}

\smallskip
We will now turn a triangulation of a polygonal domain into a finite network endowed with geometrically chosen conductances. The choice of the conductances depends on an interesting relation between the given triangulation and its dual complex. These conductances are identical to Stephenson's (see (\ref{eq:condu})), however, in this paper they are motivated by a  scheme of approximating  flux of smooth functions (see the next section)
    and the Finite Element Method (see Section~\ref{se:ENDF}).

\smallskip

Let ${\mathcal T}$ be a triangulation of a polygonal domain $\Omega$.
The induced {\it  control volumes}, or the {\it Voronoi cells} which we will  associate with a triangulation ${\mathcal T}$ are defined as follows. For each triangle $T\in {\mathcal T}$, let $c_{T}$ denote the {\it circumcenter} of $T$, which by definition is the intersection point of the perpendicular bisectors of the edges. We join $c_{T'}$ to $c_{T}$ by a segment $[c_{T'}, c_{T}]$ whenever $T$ and $T'$ share an edge.  This procedure divides each (interior) triangle $T$ into three quadrilaterals and induces a new decomposition of $\Omega$.  The star of a vertex $x \in {\mathcal T}$ is defined as the union of all edges and triangles in ${\mathcal T}$ that contain $x$ and will be denoted by $\text{Star}(x)$. The  control volume $\Omega_x$  of a vertex  $x\in {\mathcal T}$ is defined to be the star of $x$ in this new decomposition.  

\smallskip


\begin{figure}[htbp]
\begin{center}
 \scalebox{.55}{ \input{CControlVolume.pstex_t}}
 \caption{A circumcenter, the star of a vertex, and a Voronoi cell.}
\label{figure:ControlVolume}
\end{center}
\end{figure}

Let $\{ {\mathcal T}_\rho \}_{\rho>0}$ be a
family of  $\tau$-quasi-uniform  triangulations of $\Omega$ (cf. Definition~\ref{de:regular}).  Let $V_{\rho}(T)$ denote the set of vertices of $T\in {\mathcal T}_\rho$, and let $V_{\rho}^{0}({\mathcal T}_{\rho})$ denote the set of interior vertices of $V_{\rho}({\mathcal T}_{\rho})=\cup_{T\in
 {\mathcal T}_{\rho}} V_{\rho}(T)$, enumerated by $\{x_1^\rho, x_2^\rho,\ldots, x_{M(\rho)}^\rho\}$.
Each $\Omega_{x_i}$ is an open, simply connected, and polygonally  bounded set. Its boundary, $\partial \Omega_{x_i}$, consists of finitely many (straight) line segments $\Gamma_{i,j}=\partial \Omega_{x_i}\cap \partial \Omega_{x_j}$, $j=1,\ldots,n_i$, where $n_i$ is the number of vertices adjacent to $x_i$; note that along each $\Gamma_{i,j}$  the normal $\hat n|\Gamma_{i,j} = \hat n_{i,j}$ is constant.

\begin{Def}
\label{de:cond1}
Let $m_{(i,j)}$ denote the length of $\Gamma_{i,j}$, and  let $d_{ij}=|x_i - x_j|$ denote the Euclidean distance between $x_i$ and $x_j$. Then the conductance of the edge $[x_i,x_j]$ is defined by
\begin{equation}
\label{eq:cond1}
c[x_i,x_j]=\frac{m_{(i,j)}}{d_{ij}}.
\end{equation}
\end{Def}

Hence, $m_{(i,j)}$ is equal to
$|c_{T}-c_{T'} |$, where $T$ and $T'$ are the (only) two triangles that $\Gamma_{i,j}$ intersects.  Given such a  triangulation ${\mathcal T}$ of ${\mathcal A}$, following \cite[Chapter 2]{GrRoSt}, 
for each one of its Voronoi cells $\Omega_i$ which is centered at $x_i\in{\mathcal T}^{(0)}$, we  define two quantities which are determined by ${\mathcal T}$:
\begin{equation}
\label{eq:lengthofedges}
\lambda_i=\lambda_{\Omega_i} = \left( \max_{j\in N_{x_i}} m_{(i,j)}\right)^{1/2}\ \text{and }\ \lambda=\max_{x_i\in {\mathcal T}^{(0)}   } \lambda_i,
\end{equation}
where $l(\cdot)$ denotes Euclidean length.

\smallskip

In this paper, we will assume the following.

\begin{description}

\item[(V0)]
 Every triangulation ${\mathcal T}$  is $\tau$-quasi uniform for some fixed $\tau >0$ and consists exclusively of {\it nonobtuse}  triangles.  
\end{description}
 
It is well  known (see for instance \cite{Zl}) that the $\tau$-quasi uniform condition is equivalent to Zl\'{a}mal's condition: there exists a positive constant, $\theta_{\text{min}}$,  such that, for all $T\in \bigcup_\rho {\mathcal T}_\rho$, and for any angle $\theta_T$ of $T$, we have
\begin{equation}
\label{minang}
\theta_{\text{min}}\leq \theta_T.
\end{equation}

\smallskip

Assumption (V0) also implies  that  ${\mathcal T}$  is a {\it Delaunay triangulation}, i.e., no point in the vertex set of ${\mathcal T}$  lies inside the circumcircle of any triangle in ${\mathcal T}$, and the corresponding Voronoi diagrams can be constructed by means of the perpendicular bisectors of the triangles' edges (see for instance \cite[Theorem 6.5]{AnKn}).


\medskip

We now define a class of combinatorial functions that will naturally appear in the next section. 
The combinatorial counterpart of the real part of the uniformizing mapping of an annulus belongs 
to this class of discrete functions.

\begin{Def}
\label{de:asyharmonic} 
Let $\alpha \in \mathbb{R}$ be a positive constant. Let ${\mathcal T}$ be a triangulation of ${\mathcal A}$, with Voronoi cells  $\{\Omega_i\}_{i\in J}$.
A function $g :{\mathcal T}^{(0)}\rightarrow \mathbb{R}$ is said to be  asymptotically harmonic of order $\alpha$ with respect to  conductance constants  $\{c_{(i,j)}=c(x_i,x_j)\}$, if there exists a non-negative constant, $d$, such that 
\begin{equation}
\label{eq:asydishar}
 \big|\Delta g(x_i)  \big|   =   \big|\sum_{j\in N_{x_i}} c_{(i,j)} \big( g(x_j) - g(x_i)\big )  \big| \leq d \lambda^{\alpha}, \text{  for all } x_i\in {\mathcal T}^{(0)}.
\end{equation}
\end{Def}
Note that the case $d=0$ corresponds to $g$ being discrete harmonic.

\smallskip

\begin{figure}[htbp]
\begin{center}
 \scalebox{.55}{ \input{Aloop.pstex_t}}
 \caption{An illustration of the path $\gamma$  (with arrows) in  Lemma \ref{le:almostzeroflux}.}
\label{figure:fluxfollowpath0}
\end{center}
\end{figure}

The following lemma provides an estimate for the summation of the normal derivative of $g$ along {\it special} closed curves in ${\mathcal T}^{(1)}$.  This sum encapsulates the {\sl discrete flux} of $g$ through the boundary of the union of those Voronoi cells which such a closed curve encloses.

\begin{Lem}[Asymptotic flux estimate]
\label{le:almostzeroflux}
Let $g :{\mathcal T}^{(0)}\rightarrow \mathbb{R}$ be a discrete, harmonic or asymptotically harmonic of order $\alpha$, with respect to conductance constants  $\{c_{(i,j)}\}$.
Then, for any homotopically trivial (in ${\mathcal A}$), closed path $\gamma\subset{\mathcal T}^{(1)}$ which contains an integer number of 
 Voronoi cells $\Omega_{x_i}$, $x_i\in  {\mathcal T}^{(0)} $, we have
\begin{equation}
\label{eq:fluxisalmostzero}
\sum_{x\in \gamma} \frac{\partial g}{\partial n}(x) = 
      0,\  \text{ if } g \text{ is harmonic.}  
\end{equation}
Furthermore, if $g$ is asymptotically harmonic or order $\alpha$, then there exists a positive constant, $D$, such that 
\begin{equation}      
      \sum_{x\in \gamma} \frac{\partial g}{\partial n}(x) \leq 
      D \lambda^{\alpha}. 
\end{equation}
\end{Lem}

\begin{proof} Let $\Omega_m=\cup_{i=1}^{m}\Omega_{x_i}$ be the maximal collection of control volumes enclosed in $\gamma$, and let $E_m$ be those  edges  of  ${\mathcal T}^{(1)}$ that lies in the interior of the bounded region enclosed  by $\gamma$.
The first Green identity (see for instance \cite[Proposition 3.1]{BeCaEn}) implies that 
for $u,v \in {\mathcal P}(\Omega_m)$,  we have that 

\label{eq:GreenIdentity}
\begin{equation}
\label{eq:Green}
 \sum_{[i,j]\in {\bar
E}_{m}}c_{(i,j)}\big(u(i)-u(j)\big)\big(v(i)-v(j)\big)=\sum_{x\in \Omega_m\cap {\mathcal T}^{(0)}}\Delta
u(x)v(x)+\sum_{y\in\gamma}\frac{\bord u}{\bord n}(\Omega_m)(y)v(y).
\end{equation}

We now let $v \equiv 1$ in the above equality, and obtain  

\begin{equation}
\label{eq:Green1}
 0=\sum_{i=1}^{m}\Delta u(x_i)+\sum_{y\in\gamma}\frac{\bord u}{\bord n}(\Omega_m)(y).
\end{equation}

It therefore follows, by the definition of the combinatorial laplacian, that 
\begin{equation}
\label{eq:Green2}
 0=\sum_{i=1}^{m} \sum_{j\in N_i} c_{(i,j)} \big( u(j) - u(i) \big)
  +\sum_{y\in\gamma}\frac{\bord u}{\bord n}(\Omega_m)(y). 
\end{equation}
Hence, the first assertion of the lemma readily follows; the second assertion follows with $D=D(d, m)$.

\end{proof}

With the  notation of the lemma the following corollary easily follows.
\begin{Cor}[Asymptotic path independence]
\label{le:pathdepen1}
Let $\gamma_1$ and  $\gamma_2$ be two simple paths in ${\mathcal T}^{(1)}\subset{\mathcal A}$ joining two vertices $x_1,x_2\in {\mathcal T}^{(0)}$,  such that the path  $\gamma_2^{-1} \circ \gamma_1$ is trivial in $\pi_1({\mathcal A})$, and contains  an integer number of control volumes $\Omega_{x_i}$.  Then, if $g$ is harmonic we have 
\begin{equation}
\label{eq:fluxisalmostzero1}
\sum_{x\in \gamma_1} \frac{\partial g}{\partial n}(x)-\sum_{x\in \gamma_2} \frac{\partial g}{\partial n}(x)  = 
      0.
\end{equation}

Furthermore, if $g$ is asymptotically harmonic or order $\alpha$, then 
\begin{equation}
\big |\sum_{x\in \gamma_1} \frac{\partial g}{\partial n}(x)-\sum_{x\in \gamma_2} \frac{\partial g}{\partial n}(x) \big | \leq     D\lambda^{\alpha}.
\end{equation}
\end{Cor}

\bigskip

Let $g$ be a discrete, harmonic or asymptotically harmonic function. Inspired by the classical construction of the harmonic conjugate function as recalled in Definition~\ref{de:smoothconju}, we will define a {\it combinatorial conjugate} to $g$ using discrete sums, i.e, using discrete fluxes. To this end,  we will need to define a special class of paths in ${\mathcal T}^{(1)}$. Thereafter, by summing a generalized version of  the normal derivative of $g$, along a path from this class, the combinatorial conjugate function of $g$ will be  defined  {\it firstly }at the vertices of the Voronoi cells of  ${\mathcal T}$. In the next section, it will be proved that the imaginary part of a uniformizing map of a given annulus can be approximated by a sequence of combinatorial conjugate functions of $g$.

\medskip

We let  $\Lambda$ denote the union of all Voronoi cells of a given ${\mathcal T}$. An {\it interior} cell is one such that its  vertex boundary is  disjoint form $\partial{\mathcal A}$.
We now make the following:
\begin{Def}[Flux fellow paths]
\label{de:fellowpath}
Let $\omega_0$ be a fixed vertex in an interior Voronoi cell , and let $\omega$ be any vertex in an interior cell of $\Lambda$. Let 
$\gamma_{\Lambda}=[\omega_0,\dots,\omega_k=\omega]$ be a simple, piecewise linear  
curve in ${\Lambda}^{(1)}$ joining $\omega_0$ to $\omega$, whose trace is disjoint from $\partial{\mathcal A}$. For each $[\omega_i, \omega_{i+1}]$, $i=0,\ldots, k-1$, let 
$x_i$ be the vertex in ${\mathcal T}^{(0)}$ on the unique edge intersecting $[\omega_i, \omega_{i+1}]$, and  which is to the right of $[\omega_i, \omega_{i+1}]$. Then 
$\gamma_{\mathcal T}=[x_0,\ldots,x_{k-1}]\subset {\mathcal T}^{(1)}$  will be called the {\it flux fellow path} of $\gamma_{\Lambda}$ (see Figure~\ref{figure:fluxfollowpath}).
\end{Def}

\begin{figure}[htbp]
\begin{center}
 \scalebox{.45}{ \input{Voronoi.pstex_t}}
 \caption{A path $\gamma_\Lambda$ ($[\omega_0\ldots\omega_k]$) in ${\Lambda}^{(1)}$, its flux fellow path $\gamma_{\mathcal T}$ ($[x_0\ldots x_k]$) in ${\mathcal T}^{(1)}$, and edges for the flux computation (each emanates from the $x_i$'s to $[x_i,x_{i+1}]$.}
\label{figure:fluxfollowpath}
\end{center}
\end{figure}

\begin{Rem}
\label{re:itisapath}
The discussion preceding condition (V0)  grants us that $\gamma_{\mathcal T}$ is indeed a path in $ {\mathcal T}^{(1)}$ (which is disjoint from $\partial{\mathcal A}$); we orient each edge in the paths $\gamma_{\Lambda},\gamma_{\mathcal T}$ according to an increasing order of its vertices. 
\end{Rem}

Note that $\gamma_{\mathcal T}$  is uniquely determined only after a choice of $\gamma_\Lambda$ was made.   However, vertices of $\gamma_{\mathcal T}$ do {\it not} belong to the vertex boundary of any naturally  defined domain in ${\mathcal T}^{(0)}$.  In light of the coming applications,  we will now extend the notion of the discrete normal derivative (see (\ref{eq:nor}) in Definition~\ref{de:lapandnorm}). The definition of the combinatorial  conjugate function will utilize this  generalized version. 

 



\smallskip
In the definition below, we will abuse notation and use the notation for normal derivative that appeared in  Equation~(\ref{eq:nor}). 

\smallskip

\begin{Def}[Flux through edges]
\label{de:extnormderi}
For any vertex $y\in \gamma_{{\mathcal T}}$, we define 
\begin{equation}
\frac{\partial g}{\partial n}(\gamma_\Lambda,\gamma_{\mathcal T})(y)= \sum_{x\sim y} c(x,y)\big (g(x)- g(y)\big), 
\end{equation}
where the sum is taken over all those vertices  $x\in{\mathcal T}^{(0)}$ which are adjacent to $y$ along an edge which intersects  $\gamma_{\Lambda}$.
\end{Def}


\smallskip

We now make a combinatorial definition which imitates the smooth one (Definition~\ref{de:smoothconju}).   

\begin{Def}[A combinatorial  conjugate]
\label{de:combiconjugate}
Let ${\mathcal T}$ be a triangulation of ${\mathcal A}$, and let $\Lambda$ denote the union of all Voronoi cells of ${\mathcal T}$. 
Let $\omega_0$ be a fixed vertex in an interior cell of $\Lambda$, and  let  $\omega$ be any vertex of an interior cell $\Lambda$. Let 
$\gamma_{\Lambda}\subset \Lambda^{(1)}$ be a  simple, piecewise linear  
curve  joining $\omega_0$ to $\omega$ whose trace is disjoint from $\partial{\mathcal A}$.  Let $\gamma_{\mathcal T}\subset {\mathcal T}^{(1)}$ be the flux fellow path of $\gamma_{\Lambda}$. 

\smallskip

(i) Let $g$ be a discrete, harmonic or asymptotically harmonic function of order $\alpha$. Then, for every $\omega\in  \Lambda$, a (multivalued) combinatorial conjugate of $g$  is defined by 
\begin{equation}
\bar g^{*}(\omega)= \bar g^{*}(\omega_0) +\sum_{y\in \gamma_{\mathcal T}}\frac{\partial g}{\partial n}(\gamma_\Lambda,\gamma_{\mathcal T})(y).
\end{equation}
where $\bar g^{*}(\omega_0)$ is some arbitrary, fixed real constant. 

By choosing a vertex in each 2-cell in $\Lambda^{(2)}$ and drawing diagonals to its other vertices, this cell is divided into triangles with vertices in $\Lambda^{(0)}$ and disjoint interiors.  We then extend $g^{*}$ affinely over edges in $\Lambda^{(1)}$ and over triangles in $\Lambda^{(2)}$. (By abuse of notation, the extended function will also be called a combinatorial conjugate of $g$.)

(ii) Let $\alpha_\Lambda$ be any simple, counter-clockwise oriented,  closed curve in $\Lambda$
whose winding number is equal to 1.   
The period of $\bar g^{*}$ is defined by 
\begin{equation}
\label{eq:periodcomb}
\text{period}(\bar g^{*})=\sum_{\xi\in \alpha_{\mathcal T}} \frac{\partial g}{\partial n}(\xi).
\end{equation}
\end{Def}

\smallskip

\begin{Rem}
\label{re:upto}
It follows form Corollary~\ref{le:pathdepen1}  that if $g$ is discrete harmonic (i) and (ii) are independent of the choices of $\gamma_\Lambda$ and $\gamma_\tau$. If $g$ is discrete and asymptotic harmonic of order $\alpha$, then (i) and (ii) hold up to (in absolute value)  a factor of at most $D\lambda^{\alpha}$. Furthermore, combinatorial provisions analogous to those in Remark~\ref{re:pathind} hold for $\bar{g}^{*}$.

\end{Rem}

\begin{Rem}
The search for discrete analogues of conformal maps has a long and rich history. We refer to \cite{Mer} and \cite[Section 2]{ChSm} for excellent recent accounts. We should also mention that a search for a combinatorial Hodge star operator has recently gained much attention and is closely related to the construction of a harmonic conjugate function. We refer the reader to  \cite{Hi} and to \cite{Po, PoRa} for further details,  examples, and applications of such combinatorial operators. 
\end{Rem}


\section{Uniformization of a planar annulus}
\label{se:makeconf11}

In this section,
we prove the main theorem of this paper, Theorem~\ref{th:annulus2}
which provides a discrete scheme of approximation of the uniformizing map of a polygonal annulus.
Thereafter,  we will prove that the hypotheses ``polygonal boundary" in  this theorem, can be relaxed to ``continuous boundary".

\smallskip

We keep the notation of the previous sections and appendices. Let ${\mathcal A}$ be  endowed with a
family of   $\tau$-quasi-uniform triangulations  $\{ {\mathcal T}_{\rho_{n}} \}$ (cf. Definition~\ref{de:regular}) such that $\rho_{n}\rightarrow 0$, as $n\rightarrow\infty$. For each 
${\mathcal T}_{\rho_n}$, let the corresponding family of Voronoi cells be denoted by $\{\Omega_n=\Omega_{\rho_n}\}$ (see the discussion proceeding Figure~\ref{figure:ControlVolume}).  We let $V_{\rho}(T)$ denote the set of vertices of a triangle $T\in {\mathcal T}^{(2)}_\rho$, and let $V_{\rho}^{0}({\mathcal T}_{\rho})$ denote the set of interior vertices of $V_{\rho}({\mathcal T}_{\rho})=\cup_{T\in
 {\mathcal T}_{\rho}} V_{\rho}(T)$.

For each vertex $x$, recall that $N_{x}$ denotes the set of neighboring (in ${\mathcal T}^{(0)}$) vertices of $x$.
 In addition to requiring that each triangulation is quasi-uniform, henceforth in this paper,  we will assume the existence of a constant $\tau_0$ such that for all $\rho_n<\tau_0$, the following hold:
\begin{description}
\item[(V1)] The cardinality $N_x$ of each vertex $x\in {\mathcal T}_{\rho_{n}}^{(0)}$ remains uniformly bounded, that is, 
$$
\max_x \left\{ \text{card} (N_x)\right\} \leq m_{*} \text{ for some }\  m_{*}\in \mathbb{N};
$$

\item[(V2)] Each point $x_{i,j}=[x_i,x_j]\cap \Gamma_{i,j}$ is the middle point of the segment $\Gamma_{i,j}$.
\end{description}

\subsection{Stephenson's conductance constants from a  flux perspective}
\label{se:makeconf1}

In this section, we will continue to assume that ${\mathcal A}$ is a fixed polygonal annulus. We will  construct  a sequence of mappings, obtained via a refined sequence of quasi-uniform triangulations and  conductance constants along edges according to (\ref{eq:cond1}), from the interior of ${\mathcal A}$ onto the interior of a concentric Euclidean annulus. The image annulus is determined  (see Equation~(\ref{eq:dimann})) by the solution of a specific smooth boundary value problem defined on the domain: Laplace's equation with non-homogeneous boundary values. 
Theorem~\ref{th:annulus2} demonstrates that the  sequence converges uniformly on compact subsets  to a conformal homeomorphism.  

\smallskip



In the proof of  Theorem~\ref{th:annulus2}, we  will need to consider a Dirichlet boundary value problem for the Laplace equation with prescribed boundary data as formulated in (\ref{de:strong}). In fact, an approximation scheme for this type of boundary value problem is  obtained through the analysis of 
a naturally defined different boundary value problem, i.e.,  one with prescribed Poisson data and homogeneous boundary condition. 

\smallskip


In order to get the necessary analysis in place, let $h$ be the continuous function defined on $\partial{\mathcal A}$ by setting  
\begin{equation}
\label{de:theboundaryfunction}
h|E_1=1 \text{ and }  h|E_2=0,\text{ and}
\end{equation} 
let $\tilde h\in  C^{2}({\mathcal A})\cap C^{0}({\bar  A})$ be an extension of $h$  to the interior of ${\mathcal A}$ with $\Delta\tilde h\neq 0$. Recall that the existence of such an extension is a consequence of Whitney's seminal work \cite{Whi}. 



\smallskip

We now define a Poisson boundary value problem, which is naturally associated with the Laplace problem   (\ref{de:strong}) we wish to solve, by 
\begin{equation}
\label{eq:weak}
\begin{cases}
      & \Delta \tilde u = -\Delta \tilde h,\ \text{in}\ {\mathcal A}\\
      & \tilde{u}=0,\  \text{on}\  \partial{\mathcal A}.  %
\end{cases}
\end{equation}

\begin{Rem}
\label{re:PononHom}
The existence and uniqueness of a strong  solution to (\ref{eq:weak}) follows (for instance) from Corollary~\ref{co:Jordan domains}, by setting $\tilde u = u -\tilde h$, where $u$ is the strong solution of (\ref{de:strong0}).
\end{Rem}

\subsection{The convergence of the piecewise linear approximations to the  strong solution}
\label{se:ENFT}
In this section, we will recall one of the main convergence results in a classical paper by Schatz and Wahlbin  \cite{ScWa}. This foundational quantitative result, derived explicitly  by the finite element method (see \ref{se:ENDF}),  describes the rate of approximation of particular combination of piecewise linear maps to the smooth solution of a Poisson boundary value problem with homogeneous boundary values.
 
 \smallskip 

For the applications of this paper, it is necessary to consider {\it non-convex} polygonal domains. 
In order to approximate  a given Jordan domain, which in general is not convex, 
we will construct a sequence of {\it necessarily} non-convex polygonal domains where each domain is triangulated by acute triangles, and each triangulation has a uniform upper and lower bounds on their largest and smallest angles, respectively. However, due to the presence of corner singularities of vertex angles that are bigger than $\pi$, 
 the $L_\infty$ error analysis of the approximation provided by the finite element is quite subtle (see for instance the monograph \cite{Gri} for treatment of convergence in other norms). 
 
\smallskip  

Let  $\Omega$  be  a bounded, (possibly) non-convex, polygonal domain. Let $\partial \Omega$  denote the  boundary of $\Omega$. Therefore, $\partial \Omega$ consists of a finite number of straight line segments meeting at vertices $v_j$, $j=1,\ldots, M_2$, of interior angles $0< \alpha_1 \leq \cdots \leq \alpha_{M_2}
<2\pi$; let $\beta_j = \pi/\alpha_j$.  We let $\Upsilon_j$, $j=1,\cdots, M_2$, denote the intersection of $\Omega$ with a disk centered at $v_j$ and such that $\Upsilon_j$ contains no other vertex, and define  
$\Upsilon_0 = \Omega \setminus (\cup_{j=1}^{M_2} \bar \Upsilon_j)$. Then,  
it is well known that  the solution $u$ of the boundary value problem  defined in (\ref{eq:diffusion}) is {\it not} always in $H^{2}(\Omega)$ (see for instance \cite[Section 2]{Gri}). 
However, for every $\epsilon >0$, $u$  always belongs to  
a fractional order Sobolev space $H^{1+\beta_{M_2}-\epsilon}(\Upsilon_j)$ or $C^{\beta_{M_2} -\epsilon}(\bar\Upsilon_j)$, and $u\in C^{\infty}(\Upsilon_0)$ (cf. \cite[page 74]{ScWa}). 

\smallskip

The following foundational result  was obtained by Schatz and Wahlbin in the 70's. It is the main analytical  result which will be used in this paper.
 \begin{Thm}[\mbox{\cite[Theorem 4.1]{ScWa}}]
 \label{Th:FVMEtimate}
Let $\epsilon >0$ be given. Let $\tilde u$ and $\tilde u_\rho$ be the solutions of  (\ref{eq:diffusion}) and (\ref{eq:intalongvolumes1}), respectively, with $f\in L_{p},\   p>1$. Then there exists a constant $c=c(f,\epsilon)$ such that for $\rho$ sufficiently small 
 \begin{equation}
 \label{eq:Linftyestimates}
 \|\tilde u-\tilde u_\rho\|_{ L_{\infty}(\Upsilon_0)} \leq c \rho^{\min(2, 2\beta_{M_2})-\epsilon}.
 \end{equation}
 \end{Thm}
Since the polygonal domains in the applications of this paper are not convex, $ 1/2 < \beta_{M_2} $, hence, $\min{(2, 2\beta_{M_2})}=2\beta_{M_2}>1$.

\begin{Rem}
There is another interesting part to this theorem which provides an $L_\infty$ estimate inside $\Upsilon_j,  j=1,\ldots, M_2$ (we will not use this part in this paper). 
\end{Rem}

Finally, let  $u$ 
be a solution of (\ref{de:strong}), hence, we may write  $u=\tilde{u} + \tilde{h}$. Thus, since  with  $\tilde u$ we are in the framework of (\ref{eq:diffusion}), we can proceed with 

\begin{Def}
\label{de:approxdisc}
Let $\tilde u$ and $\tilde u_n=\tilde u_{\rho_n}$ be the solutions of (\ref{eq:weak}) and 
(\ref{eq:intalongvolumes1}) with $f =  -\Delta \tilde h $, respectively. We will also assume that for every $n>0$, $\tilde u_n$  is presented by a linear combination of the basis elements in $\mathbb{V}_{0,{\mathcal T}_{\rho_n}}$, as  described in (\ref{eq:coefficients}). 
\end{Def}



We also need a natural way to discretize $\tilde h$. Hence,  let us denote  the projection of 
$\tilde h$ on $ {\mathcal T}_{\rho_{n}}^{(0)}$ by $\Pi_n(\tilde h)$, that is,  
\begin{equation}
\label{eq;projection}
\Pi_n(\tilde h)(x) =\tilde h (x), \text{ for every } x\in {\mathcal T}_{\rho_{n}}^{(0)},
\end{equation}  
and then extend afinely over edges and  triangles.

\smallskip

As the final preparation for the proof of our main theorem, let us recall  a lemma due to Grossmann, Roos and Stynes. This important lemma provides an approximation of the integral of 
the Laplacian of a smooth function, given for instance by the right-hand side of
(\ref{eq:weak}), by a discrete quantity - a finite difference expression which utilizes Stephenson's constants. Recall that  we have let  $m_{(i,j)}$ denote the length of $\Gamma_{i,j}$, and  $d_{ij}=|x_i - x_j|$ denote the Euclidean distance between $x_i$ and $x_j$ (see Definition~\ref{de:cond1}). The conductance of the edge $[x_i,x_j]$ is then defined by
$c[x_i,x_j]=\frac{m_{(i,j)}}{d_{ij}}.$
\begin{Lem}[\mbox{\cite[Lemma 2.63]{GrRoSt}}]
\label{eq:Poissonnonhomboundary} Let $\Omega\subset \mathbb{R}^2$ be a convex polygon.
Assume that conditions (V1) and (V2) hold. Assume that  the solution of the Poisson boundary value problem with (possible) non-homogenous boundary data 
\begin{equation}
\label{eq:weak1}
\begin{cases}
      & \Delta  w = f,\ \text{in}\ \Omega\\
      & w=g,\  \text{on}\  \Gamma=\partial\Omega,  %
\end{cases}
\end{equation}
belongs to $C^2(\bar\Omega)$. 
 Then there exists some constant $c=c(w,\Omega)$ such that
\begin{equation}
\label{eq:estimateforconj}
 \Bigg|  \sum_{x_j\sim x_i, j\neq i} \frac{m_{(i,j)} }{d_{ij}}  \big(  w(x_j) - w(x_i)  \big) - \int_{\Omega_{x_i}} f dx    \Bigg|\leq c \lambda_i^3, 
\end{equation}
 for all interior vertices $x_i$, and where $\lambda_i$ is defined by (\ref{eq:lengthofedges}).
\end{Lem}

Thus, this  lemma provides an estimate for the discrete flux along the {\it full} boundary of {\it one} Voronoi cell  of a solution of  the boundary value problem 
(\ref{eq:weak1}).  In the statement of this lemma, $\Omega$ is assumed to be a convex polygon in $ \mathbb{R}^{2}$. However, the proof remains valid even if the {\it regularity}  assumption of the solution is only assumed to hold for any close, proper subset  of $\Omega$ with {\it thick} enough neighborhood; that is, if the subset and its neighborhood are still contained in $\Omega$.  This weaker regularity assumption will be used in the proof of Theorem~\ref{th:annulus2}, where we will also show how to choose an appropriate thick neighborhood. Assumptions (V1) and (V2)  are critical to the proof of this Lemma.

\begin{Rem} Equation~(\ref{eq:overatriangle})  complements (\ref{eq:estimateforconj}) in regards to the terms appearing in the conductances and exploits the connection to Stephenson's conjecture (Conjecture~\ref{con:step}) from the finite element perspective. 

\end{Rem}

\subsection{The main theorem}
\label{se:main}

\smallskip
With the notation above, we now turn to the main theorem of this paper.
In order to ease the  notation, we will not distinguish between a map defined on the $0$-skeleton of a triangulation and the affine extension of the map on edges and triangles. Finally, for every $n$,  let ${\mathcal T}_n={\mathcal T}_{\rho_n}$.

\begin{Thm}
\label{th:annulus2}  
Let  $\{{\mathcal T}_{n}\}$ be a sequence of quasi-uniform triangulations of ${\mathcal A}$ of mesh size $\rho_n\rightarrow 0$ as $n\rightarrow \infty$, and let the corresponding family of Voronoi cells of each ${\mathcal T}_n$ be denoted by $\{\Omega_n\}$. Assume in addition that $\{{\mathcal T}_{n}\}$  satisfies conditions $(V1)$ and $(V2)$. 
  Let the conductance of each edge $e\in T$, $T\in {\mathcal T}^{(2)}_{n}$ be defined  by 
\begin{equation}
\label{eq:condu0}
c_n(e) = \frac{m_{(i,j),n}^{T}}{d_{ij,n}}.
\end{equation}

\smallskip

Let $u$ and $h$ be given in $(\ref{de:strong})$ and $(\ref{de:theboundaryfunction})$, respectively, and define (see Definition~(\ref{de:approxdisc}))
\begin{equation}
\label{eq:theapproximants}
g_n= \tilde u_n + \Pi_n(\tilde h).
\end{equation}
 
\smallskip
 
Then,  as $n\rightarrow \infty$ the following assertions hold:

\begin{enumerate}

\item   $\| u - g_n\|_{L_\infty({\mathcal A})}\rightarrow 0$.

\smallskip

\item On each proper, compact subset of ${\mathcal A}$, the $g_n$'s  are asymptotically harmonic of order $\alpha=3$.

\smallskip

\item Let $\bar g_n^{*}$ denote a suitable normalized combinatorial conjugate of $g_n$, and let 
   $\phi_n$ be the sequence of discrete mappings defined by extending affinely over $\Omega_n$ the sequence of discrete mappings given by
\begin{equation}
\label{eq;defofmaps}
\phi_n(\omega)=\exp\big(\frac{2\pi}{\mbox{\rm period}(\bar g_n^{\ast})}\big( g_n(\omega)   + i \bar g_n^{*}(\omega) )\big),\ \omega\in {\mathcal A}\cap \Omega_n^{(0)}.
\end{equation}
Then a subsequence of  $\{\phi_n\}$ converges uniformly on compact subsets of ${\mathcal A}$ to a conformal homeomorphism, denoted by $\Psi_{\mathcal A}$, onto the interior of the concentric Euclidean annulus ${\mathcal E}_{\mathcal A}$, whose inner and outer radii are given by 

\begin{equation}
\label{eq:dimann}
\{R_1,R_2\}=  
\{1,  \exp\big( \frac{2\pi}{\mbox{\rm period}(u^{*})}\,     \big)\}, 
\end{equation}
where $u^{*}$ and \text{period}$(u^{*})$ are given in Definition~\ref{de:smoothconju}.
\end{enumerate}
\end{Thm}

\begin{Rem}
By choosing conductance constants according to (\ref{eq:condu0})  for every $\rho_n>0$, as predicted in Stephenson's Conjecture (see Conjecture~\ref{con:step}), each ${\mathcal T}_{\rho_n}$ is turned into a finite electrical network;  where for each $n>0$, the homogeneous part of the induced potential function, $\tilde u_n$  (see \cite{So}), satisfies the system of equations described by (\ref{eq:linearsystem}). We remark that  since for each $\rho_n>0$, the values of $u_n=u_{\rho_n}$ at the boundary vertices in  $\partial {\mathcal T}_{n}^{(0)}\subset \partial \Omega$ are given,  there is no need to specify conductance  constants for edges that are contained in $\partial\Omega$; or one can choose arbitrary values. 
\end{Rem}
The proof is not short and will therefore be  naturally divided into two parts. In the first part assertions (1) and (2) will be proved. In the second and longer part, the proof of assertion (3) which depends on (1) and (2), will be given.
 
\smallskip
 
\noindent\dem 

\smallskip

\noindent\underline{\emph{The proofs of assertions (1) and (2).}}
By letting $f= -\Delta \tilde h$ in Theorem~\ref{Th:FVMEtimate} (Schatz-Wahlbin \cite[Theorem 4.1]{ScWa}), we know that for the approximation of $\tilde u$ by $\tilde u_n$ (see (\ref{eq:weak})), the following estimate holds. 
Let $\epsilon >0$ be chosen so that $2\beta_{M_2}-\epsilon =1 + \epsilon_0$ with $\epsilon_0>0$. Let $\tilde u$ and $\tilde u_\rho$ be the solutions of  (\ref{eq:diffusion}) and (\ref{eq:intalongvolumes1}), respectively, with $f\in L_{p},  p>1$.  Since $1>\beta_{M_2} >1/2$,  the assertion of the theorem is that there exists a constant $C=C(f,\epsilon)$ such that for $\rho$ sufficiently small 
 \begin{equation}
 \label{eq:Linftyestimates0}
 \|\tilde u-\tilde u_\rho\|_{ L_{\infty}(\Upsilon_0)} \leq C \rho^{\min(2, 2\beta_{M_2})-\epsilon}=C\rho^{1+\epsilon_0}.
 \end{equation} 
Hence, as $\rho_n\rightarrow 0$ we have 
\begin{equation}
 \label{eq:Linfty1}
\|\tilde u-\tilde u_n\|_{L_{\infty}({\Upsilon_0})} \rightarrow 0. 
 \end{equation}
 
\smallskip

Since $\tilde h$ is sufficiently smooth in $\Upsilon_0$, it is well known that there exists a constant $C_1=C_1(f, \Upsilon_0)$ such that 
\begin{equation}
\label{eq:hgoingto}
\|\tilde h-\tilde  \Pi_n(\tilde h)\|_{L_{\infty}(\Upsilon_0)}\leq C_1 \rho^2, 
\end{equation}
where $\Pi_n(\tilde h)$ is the affine  interpolation of  $h$ (see (\ref{eq;projection})).
Hence,  as $\rho_n \rightarrow 0$
 we have 
 \begin{equation}
\label{eq:hgoingto1}
\|\tilde h-\tilde  \Pi_n(\tilde h)\|_{L_{\infty}(\Upsilon_0)}\rightarrow 0.
\end{equation}

\smallskip

We now show that the $g_n$'s  comprise our desired approximations to $u$ - the strong solution of the smooth Dirichlet problem for the Laplace equation (\ref{de:strong}). Indeed,  we have  
\begin{equation}
\begin{array}{ccl}
      \| u - (\Pi_n(\tilde h) + \tilde u_n)\|_{L_\infty({\Upsilon_0})}& = &\|u - \tilde h +\tilde h - (\tilde u_n+ \Pi_n(\tilde h ) )  \|_{L_\infty(\Upsilon_0)}  \\ \\
      & = &\|\tilde u + \tilde h - (\tilde u_n + \Pi_n(\tilde h )) \|_{L_\infty(\Upsilon_0)} \\ \\  
      & = &\|(\tilde u - \tilde u_n)  +(\tilde h  - \Pi_n(\tilde h )) \|_{L_\infty(\Upsilon_0)}\\ \\
      & \leq &\|\tilde u - \tilde u_n \|_{L_\infty(\Upsilon_0)} + \| \tilde h - \Pi_n(\tilde h )\|_{L_\infty(\Upsilon_0)}.
\end{array}
\end{equation}

Therefore, assertion (1)  of the Theorem is proved by applying Equations~(\ref{eq:Linfty1}) and (\ref{eq:hgoingto}), where in fact, the rate of convergence is at least of the following order in $\rho$:
\begin{equation}
\label{eq:rateofconv}
\|u-g_n \|_{L_\infty({\Upsilon_0})}     = \| u - (\Pi_n(\tilde h) + \tilde u_n)\|_{L_\infty({\Upsilon_0})}\leq C_2 \rho^{1+\epsilon_0},
\end{equation}
for some constant $C_2=C_2(C,C_1)$.

\smallskip

We now continue and prove assertion  (2) by showing that for all $n$ large enough, each $g_n$  is asymptotically harmonic of order $3$. We already observed that  $\tilde u+ \tilde h$ is a solution of (\ref{de:strong})-the Dirichlet non-homogeneous boundary value problem for the Laplace equation with boundary values prescribed  by $h|\partial{\mathcal A}$.

\smallskip
 
We now apply Lemma~\ref{eq:Poissonnonhomboundary}  with  $u=\tilde u+ \tilde h$ and $f=0$. It then follows that the following holds for all $n>0$ 
\begin{equation}
\label{eq:fluxtolaplacian}
 \Bigg|  \sum_{x_j\sim x_i, j\neq i} \frac{m_{(i,j),n} }{d_{ij,n}}  \big(  u(x_j) - u(x_i)  \big)  \Bigg|\leq C_3 \lambda_{i,n}^3,\ \text{ with }
 C_3=C_3(u,\Upsilon_0).
\end{equation}
 
By applying the law of sines in Euclidean geometry and the existence of a uniform lower bound  on the smallest angle  (see (\ref{minang}))
in the sequence $\{ {\mathcal T}_n\}$, it is easy to see that the conductances $c_n$ defined in (\ref{eq:condu0}) are uniformly bounded from above, with a bound depending {\it only} on $\theta_{\text{min}}$.
Hence, we finish the proof  of (2) by applying the triangle inequality 
and  assertion (1), in  (\ref{eq:rateofconv}).



\medskip
\noindent\underline{\emph{The proof of assertion (3).}}
We now turn to prove the uniform convergence of the $\bar g_n^{*}$'s, over compact subsets of ${\mathcal A}$,  to $u^{*}$- the smooth harmonic conjugate of $u$. In particular, we will  describe the normalization needed in assertion (3) of the theorem.  Let ${\mathcal A}^{\kappa}\subsetneq \Upsilon_0\subset{\mathcal A}$ be a compact annulus with {\it smooth} boundary which is concentric with ${\mathcal A}$, where $\kappa=\text{dist}(\partial{\mathcal A}, \partial{\mathcal A}^{\epsilon})$ is small. Let $\omega_0$ be a fixed point in ${\mathcal A}^{\kappa}$, and we set $u^{*}(\omega_0)=0$. Let $\omega$ be another fixed (for the moment) point in ${\mathcal A}^{\kappa}$.

\smallskip

Let us choose $N$ large enough so that for all $n>N$  (i.e., $\rho_n$  small enough) there exists a triangulation ${\mathcal T}_{\rho_n} \in \{ {\mathcal T}_n\}$ satisfying the following. 

\begin{description}
\item[(V3)] There exists a subset, $J'_n$, of the set of interior vertices $V_{\rho_n}^{0}({\mathcal T}_{\rho_n})$ such that 
the corresposnding  volume elements $\{\Omega_{x_i, \rho_n}\}_{x_i\in J'_n}$
  is contained in ${\mathcal A}^{\kappa}$, and  the combinatorial one vertex neighborhood of this subset, when considered in  ${\mathcal T}_{\rho_n}^{(0)}$, is also contained in ${\mathcal A}^{\kappa}$. 
\end{description}

\smallskip

For each $n$ as above, following Definition~\ref{de:fellowpath} and the discussion following it, 
we  choose one of the vertices of a cell $\Omega_{x_i,\rho_n}$ with $x_i \in J'_n$, which is closest to $\omega_0$. This vertex will  be  denoted by $\omega_{0}^n$. Let $\omega^{n}$ be any vertex in the union 
$\Lambda_{\rho_n}= \cup_{x_i} \Omega_{x_i,\rho_n}$ which is the closest to $\omega$.  Let $\gamma_{[\omega_0^n,\omega^n]}^{\rho_n}=[ \omega_0^n,\omega_1^n,\ldots,\omega_{k-1}^n,\omega^n]$ be a (piecewise linear) simple path in the one skeleton of $ \Lambda_{\rho_n}$ which connects  $\omega_0^n$ to $\omega^n$.  Note that $k$ is also a function of $n$.

\smallskip

It then follows from Definition~\ref{de:smoothconju} and 
Remark~\ref{re:pathind}, that the value of the {\it smooth} harmonic conjugate function $u^{*}$ with base at $\omega_{0}^n$, is given by  
\begin{equation}
\label{eq:valueof smooth}
u^{*}(\omega) =\int_{ \gamma_{\Lambda_{\rho_n}}[\omega_0^n,\omega^n] } \frac{\partial u}{\partial \hat n} ds +  \int_{\omega_0}^{\omega_0^n} \frac{\partial u}{\partial \hat n} ds  + \int_{\omega^n}^{\omega} \frac{\partial u}{\partial \hat n} ds,
\end{equation}
where the second integral is taken along any piecewise smooth curve joining $\omega_0$ to $\omega_0^n$, and the third integral is taken along any piecewise smooth curve joining $\omega^n$ to $\omega$.
\smallskip

We now follow the notation in Definition~\ref{de:fellowpath}, and we
let 
$\gamma_{ {\mathcal T}_{\rho_n}  }$ denote \underline{the} flux fellow path of  $\gamma_{ \Lambda_{\rho_n}} [\omega_0^n,\omega^n]$. Let us write  $\gamma_{ {\mathcal T}_{\rho_n}  }=[x_{0}^{\rho_n},\ldots,x_{k-1}^{\rho_n}]$ and let $\bar g_{n}^*(\omega_0^n)=0$. Hence, by  Definition~\ref{de:combiconjugate}, the value of  the corresponding  combinatorial conjugate of $g_n=g_{\rho_n}$ is defined at $\omega^n$ by 
\begin{equation}
\label{eq:approconj}
\bar g_{n}^*(\omega^n)=\sum_{ x\in \gamma_{ {\mathcal T}_{\rho_n}}  } \frac{\partial g_n}{\partial n}( \gamma_{\Lambda_{\rho_n}},\gamma_{ {\mathcal T}_{\rho_n}}  )(x).
\end{equation}
Note that as explained in Remark~\ref{re:upto},
 this value may change  by (up to) $D \lambda^{3}_n$, if a different  path connecting $\omega_0^n$ to $\omega^n$ and afterwards a different flux fellow path are chosen. Recall that   $D=D(n,m)$ with $m$ being the maximal number of Voronoi cells which belong to  $\Lambda_{\rho_n}$   and are contained in ${\mathcal A}^{\kappa}$. We will address this issue again after completing the following analysis.
 
\smallskip

We now turn to proving that as $ n\rightarrow \infty$, 
\begin{equation}
\label{eq:estimatediff}
|u^{*}(\omega)  -  \bar g_n^*(\omega)|\rightarrow 0   
\end{equation}
uniformly in  ${\mathcal A}^{\kappa}$.

\smallskip

As $n\rightarrow\infty$,  $\cup_i \Omega_{x_i,\rho_n}^{(0)}$ with $x_i\in J'_n$ comprises  a dense subset of ${\mathcal A}^{\kappa}$. In particular, by choosing  $\omega_0^n \rightarrow \omega_0$ and $\omega^n\rightarrow \omega$ as $n\rightarrow\infty$,  and due to the uniform continuity of the second and third  integrals in (\ref{eq:valueof smooth}) and the $\bar  g_{n}^*$ in ${\mathcal A}^{\epsilon}$,  we only need to bound from above  the difference 
\begin{equation}
\label{eq:estimatediff1}
  |u^{*}(\omega^n)  -  \bar g_n^*(\omega^n)|  =\Big | \int_{ \gamma_{\Lambda_{\rho_n}}[\omega_0^n,\omega^n]}\frac{\partial u}{\partial \hat n} ds - \sum_{ x\in \gamma_{ {\mathcal T}_{\rho_n}   }} \frac{\partial g_n}{\partial n}( \gamma_{\Lambda_{\rho_n}},\gamma_{ {\mathcal T}_{\rho_n}}  )(x)\Big |.
\end{equation}

To this end, following the definition of flux through edges, Definition~(\ref{de:extnormderi}), and the first assertion of the theorem, we will first show that 
we can replace each   $ \frac{\partial g_n}{\partial n}( \gamma_{\Lambda_{\rho_n}},\gamma_{ {\mathcal T}_{\rho_n}}  )(x)$    in the second  integrand in (\ref{eq:estimatediff1}),  by $ m_{(i,j),n} \dfrac{\partial u}{\partial \hat n}_{(i,j),n}(x_{i,j}^n)$.

\smallskip

As we noted before, $u\in C^{2}({\mathcal A}^{\kappa})$, and  therefore Equation (5.11) in \cite[Chapter 2]{GrRoSt} shows that for $x_i^n=x_{i,\rho_n}, x_j^n=x_{j,\rho_n}$, $\Gamma_{i,j}^n=[\omega_i^n,\omega_j^n]$ and $n_{i,j}^n$ its outward pointing normal unit vector,   and with $\lambda_{i,n}=\lambda_{i,\rho_n}$ (see (\ref{eq:lengthofedges})), there exists a positive constant $C_0=C_0(u, {\mathcal A}^{\kappa})$ so that  
\begin{equation}
\label{eq:normalderivative}
\Big | \dfrac{1}{d_{ij,n}} \big(u(x_j^n) - u(x_i^n)\big) - \frac{\partial u}{\partial \hat n}_{(i,j),n} (x_{i,j}^n) \Big |\leq C_0 \lambda_{i,n}^2,\  x_i\in J'_{n}, x_j\in N_{x_i,n},
\end{equation}
where $N_{x_i,n}$ denotes the set of neighbors of $x_i^n$ in  ${\mathcal T}^{(0)}_{\rho_n}$.
\smallskip

Let $T$ be any triangle in a triangulation in $\{{\mathcal T}_n\}$. As we argued in the proof of part (2), we have a uniform upper bound on the ratios $m_{(i,j),n}/d_{ij,n}$ in $\{{\mathcal T}_{\rho_n}\}$, where the bound depends {\it only} on $\theta_{\text{min}}$. Hence, by applying the triangle inequality,  multiplying Equation~(\ref{eq:normalderivative}) by $m_{(i,j),n}$, and assertion (1), we obtain that the bound 
\begin{equation}
\label{eq:replace}
\big |   \frac{ m_{(i,j),n}}{d_{ij,n}}(g_n(x^n_j)-g_n(x^n_i))  - m_{(i,j),n} \frac{\partial u}{\partial \hat n}_{(i,j),n}(x_{i,j}^n)    \big |\leq C_2 \rho_n^{1+\epsilon_0}+ C_0   \lambda_{i,n}^3,
\end{equation}
which implies that the left hand-side of (\ref{eq:replace}) converges to zero uniformly in 
${\mathcal A}^{\kappa}$.

\smallskip
For all $x_i\in J_n^{'}, x_j \in N_{x_{i,n}}$, let the continuous linear functional $T_{i,j}^n=T^n_{\Gamma_{i,j}}$ on $C^3(\bar \Omega_{x_i})$ defined by 
\begin{equation}
\label{eq:conjestimatepervolume}
T_{i,j}^{n}(u)=   \int_{\Gamma_{i,j}^n}\frac{\partial u}{\partial \hat n}_{(i,j),n} \  ds -      
m_{(i,j),n}\frac{\partial u}{ \partial \hat n}_{(i,j),n} (x_{i,j}^n).
\end{equation}
The proof leading to Equation (5.15) in \cite[Chapter 2]{GrRoSt} shows that the  following estimate holds
\begin{equation}
\label{eq:greedyoperator}
|T^{n}_{i,j}(u)|   \leq C_4 \lambda_{i,n}^{3},, \text{  where } C_4=C_4(u,\Omega_{x_i,\rho_n}).
\end{equation}
  
For $W \subset {\mathcal A}$ and any two points $a,b\in W$, we let $D(a,b)$ denote the pseudo-distance on $W$ defined as the infimum of the Euclidean lengths of curves in $W$ that join $a$ to $b$. We then define the {\it intrinsic diameter} of $W$ by  
\begin{equation}
\label{eq:Intrinsic}
\text{IDiam}(W)= \sup\{ D(a,b)\ |\  a,b\in W\}.
\end{equation}

\smallskip

We will now estimate (where  by (\ref{eq:lengthofedges})  $\lambda_n=\max_{x_i\in {\mathcal T}_{n}^{(0)}}\lambda_{i,n}$) 
\begin{equation}
\label{eq:conjestimate}
\begin{array}{cll}
\displaystyle \Big | \int_{ \gamma_{\Lambda_{\rho_n}}[\omega_0^n,\omega^n]}\frac{\partial u}{\partial \hat n} ds - \sum_{i=0}^{k(n)} m_{(i,i+1),n}\frac{\partial u}{\partial \hat n}(x_{i,i+1}^n)\Big | & = &\displaystyle \Big|  \sum_{i=0}^{k(n)} \big( \int_{[\omega_i^n,\omega_{i+1}^n]} \frac{\partial u}{\partial \hat n}_{(i,i+1),n} \   ds -  m_{(i,i+1),n}\frac{\partial u}{\partial \hat n}_{(i,i+1),n} (x_{i,i+1}^n)  \big) \Big |\\ \\ 
&\leq &\displaystyle  \sum_{i=0}^{k(n)}\Big |  \big(  \int_{[\omega_i^n,\omega_{i+1}^n]} \frac{\partial u}{ \partial \hat n}_{(i,i+1),n}  \ ds -      m_{(i,i+1),n}\frac{\partial u}{\partial \hat n}_{(i,i+1),n}  (x_{i,i+1}^n) \big)     \Big |\\ \\
&\leq & \displaystyle  \frac{ \text{IDiam}({\mathcal A}^{\kappa})}{\lambda_n} c_2(u, {\mathcal A}^{\kappa})\lambda_{i,n}^3
\text{ (since }k(n)\leq \displaystyle \frac{ \text{IDiam}({\mathcal A}^{\kappa})}{\lambda_n}) \\ \\ 
& \leq &  \displaystyle{C_4(u,{\mathcal A}^{\kappa}) \lambda_n^{-1} \lambda_{i,n}^3}   \\ \\
&\leq &\displaystyle{ C_4(u, {\mathcal A}^{\kappa})) \lambda_n^{2}.} 
 \end{array}
\end{equation}
It is evident that  as $n\rightarrow 0$  we have that  $\lambda_n\rightarrow 0$, hence, the left hand-side of (\ref{eq:conjestimate}) converges uniformly to zero in ${\mathcal A}^{\kappa}$.  We are now able to obtain the desired upper bound for (\ref{eq:estimatediff1}). Indeed, 

\begin{equation}
\label{eq:finally1}
\begin{array}{ll}
\displaystyle
 \Big | \int_{ \gamma_{\Lambda_{\rho_n}}[\omega_0^n,\omega^n]}\frac{\partial u}{\partial \hat n} ds - \sum_{ x\in \gamma_{ {\mathcal T}_{\rho_n}   }} \frac{\partial g_n}{\partial n}( \gamma_{\Lambda_{\rho_n}},\gamma_{ {\mathcal T}_{\rho_n}}  )(x)\Big | \leq   \displaystyle \big |\int_{ \gamma_{\Lambda_{\rho_n}}[\omega_0^n,\omega^n]}\frac{\partial u}{\partial \hat n} ds   -   \sum_{i=0}^{k(n)} m_{(i,i+1),n}\frac{\partial u}{\partial \hat n}_{(i,i+1),n}  (x_{i,i+1}^n) \big |  &  \\ \\
+\ \displaystyle  \big|  \sum_{i=0}^{k(n)} m_{(i,i+1),n}\frac{\partial u}{\partial \hat n}_{(i,i+1),n}  (x_{i,i+1}^n) - \sum_{ x\in \gamma_{ {\mathcal T}_{\rho_n}   }} \frac{\partial g_n}{\partial n}( \gamma_{\Lambda_{\rho_n}},\gamma_{ {\mathcal T}_{\rho_n}}  )(x)
\big |.   
&  
\end{array}
\end{equation}

Since a bound on the first term in the left hand-side is established in (\ref{eq:conjestimate}), we only need to bound uniformly from above the second term in (\ref{eq:finally1}).  To this end, using the estimate in (\ref{eq:replace}) we can write 

\begin{equation}
\label{eq:finally2}
 \big|  \sum_{i=0}^{k(n)} m_{(i,i+1),n}\frac{\partial u}{\partial \hat n}_{(i,i+1),n}  (x_{i,i+1}^n) - \sum_{ x\in \gamma_{ {\mathcal T}_{\rho_n}   }} \frac{\partial g_n}{\partial n}( \gamma_{\Lambda_{\rho_n}},\gamma_{ {\mathcal T}_{\rho_n}}  )(x)
\big |\leq \displaystyle\frac{ \text{IDiam}({\mathcal A}^{\epsilon})}{\lambda_n}( C_2 \rho_n^{1+\epsilon_0}+ C_0   \lambda_{i,n}^3).
\end{equation}
The uniform lower bound on the smallest angle implies that $\rho_n\leq C_5(\theta_{\text{min}})\lambda_n$, hence we have 
\begin{equation}
\label{eq:finally3}
 \frac{\text{IDiam}({\mathcal A}^{\kappa})}{\lambda_n}( C_2 \rho_n^{1+\epsilon_0}+ C_0   \lambda_{i,n}^3)
 \leq   \displaystyle C_5(u,{\mathcal A}^{\kappa})\lambda_{n}^{-1}\rho_{n}^{1+\epsilon_0}+   C_{4}(u,{\mathcal A}^{\kappa}) \lambda_{n}^{-1} \lambda_{i,n}^{3}\leq K(u,{\mathcal A}^{\kappa}) (\rho_n^{\epsilon_0}+\lambda_{n}^2).
\end{equation}

To end this part of the argument (as we mentioned after (\ref{eq:approconj})) we now assume that a different  path connecting $\omega_0^n$ to $\omega^n$ is chosen. This will  yield  a different choice of a flux fellow path and a new normalized conjugate function. However, we may use the triangle inequality, estimate  (\ref{eq:conjestimate}), and the fact that 
$n\leq  \frac{ \text{IDiam}({\mathcal A}^{\kappa})}{\lambda_n} $ to obtain the same kind of estimate for the new conjugate function.

\smallskip


 




We also need to prove that 
\begin{equation}
\label{eq:convofper}
\mbox{\rm period}(\bar g_{n}^{*}) \rightarrow \mbox{\rm period}(u^{*}).
\end{equation}
To this end, let us choose a point $P_0$ in ${\mathcal A}^{\kappa}$, and let $\beta$ and 
$\text{period}(u^{*})$  be given according to Definition~\ref{de:smoothconju}. Furthermore, let
$p_n\in \Lambda_{\rho_n}$ be chosen so that  $p_n\rightarrow P_0$. Let 
$\gamma_{{\mathcal T}_n}$ be a closed curve in $\Lambda_n^{(1)}$, based at $p_n$ according to which $\mbox{\rm period}(\bar g_{n}^{*})$ is computed. 

Since $u^{*}$ is continuous in ${\mathcal A}^{\kappa}$, we have 
\begin{equation}
\label{eq;contiofconj}
u^{*}(p_n) \rightarrow u^{*}(P_0) \text{ as } n\rightarrow\infty.
\end{equation}

By applying the analysis leading to the estimate in (\ref{eq:estimatediff}) with  $p_n$ 
replacing  $\omega_n$,  we now conclude that  (\ref{eq:convofper}) holds.

It now follows that (up to choosing a subsequence) the $\phi_n$'s converge uniformly on compact subsets of ${\mathcal A}$ to 
\begin{equation}
\label{eq:confhomeo}
\Phi_{\mathcal A}(z) = \exp\big(\frac{2\pi}{\mbox{\rm period}( u^{\ast})}\big( u(z)   + i u^{*}(z) )\big).
\end{equation}

We end the proof by recalling a classical result (see for instance \cite[Section 7]{Cou} or \cite[Theorem 4.3]{Tay}) which asserts that  $\Phi_{\mathcal A}$  is a conformal homeomorphism between the interiors of 
${\mathcal A}$ and  ${\mathcal E}_{\mathcal A}$, respectively. 

\eop{Theorem~\ref{th:annulus2}}

\subsection{The case of continuous boundary.}
\label{se:anncont}
In this paragraph, we will briefly indicate why the boundary regularity assumption  in Theorem~\ref{th:annulus2} can be relaxed.  Assume that ${\mathcal A}$ is a planar annulus, where $\partial\mathcal A$ is a union of  disjoint, Jordan  curves.  

\begin{Def}[\mbox{\cite[I.6.7]{LeVi}}]
\label{de:conveinside}
A sequence of planar annuli  ${\mathcal R}_j\subset {\mathcal R}$, $j=1,2,\ldots$,  with $\{C_j^1,C_j^2\}$ as the components of their complements, {\it converges from the inside} to an annulus ${\mathcal R}$ with $\{{\mathcal R_1}, {\mathcal R_2}\}$ as components of its complement, if the following holds:  for every $\epsilon>0$ there exists 
$n_{\epsilon}$ such that for $n\geq n_{\epsilon}$ every point of $(C_j^i)_{i=1,2}$ lies within a spherical distance  less than $\epsilon$ of the set $({\mathcal R_i}, {\mathcal R_2})_{i=1,2}$. 
\end{Def}
A classical construction due to Kellogg \cite[Chapter XI.14]{Kel} grants us an existence of 
 a nested sequence of annuli, $\{\mathcal A_i\}$, where for all $i>0$, $\{\mathcal A_i\}\subsetneq {\mathcal A}$,  the boundary of 
 ${\mathcal A}_i$ is polygonal, and the sequence converges to ${\mathcal A}$ from the inside. Furthermore, since each ${\mathcal A}_i$ is made of a lattice of squares, it is easy to construct a sequence of qausi-uniform triangulation of each  ${\mathcal A}_i$,  where each triangulation satisfes  assumptions (V0)-(V3). Thus, ${\mathcal A}$  is presented as an increasing union of open subsets. The interiors of the ${\mathcal A}_i$  and each conformal embedding $\Phi_{{\mathcal A}_i}$ can be approximated according to Theorem~\ref{th:annulus2}. It follows that, up to normalization of the maps $\Phi_{{\mathcal A}_i}$, a subsequence of the $\{\Phi_{{\mathcal A}_i}\}$
will converge uniformly on compact subsets of ${\mathcal A}$, to its uniformizing map (see for instance \cite[Lemma 2.2]{Si} or \cite[Page 223 ]{Ca}). Hence, we have the following

\begin{Cor}
\label{th:anncont}
With  the additional approximation processes described in the paragraph above, we may assume in Theorem~\ref{th:annulus2} that $\partial {\mathcal A}$ is continuous.
\end{Cor}

The following figures depict the evolution of two polygonal annuli under the approximation scheme provided  in Theorem 3.13.

\begin{figure}[h]
 \begin{minipage}[t]{.45\textwidth}
    \begin{center}  
      \includegraphics[width=2.25in]{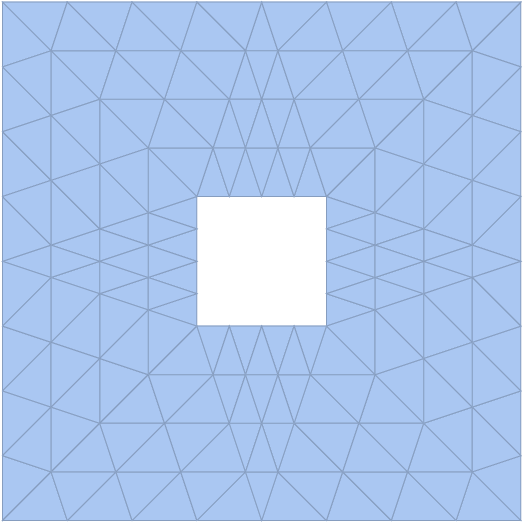}
      \label{polyann}
    \end{center}
  \end{minipage}
 \hfill 
 \begin{minipage}[t]{.45\textwidth}
    \begin{center}  
      \includegraphics[width=2.25in]{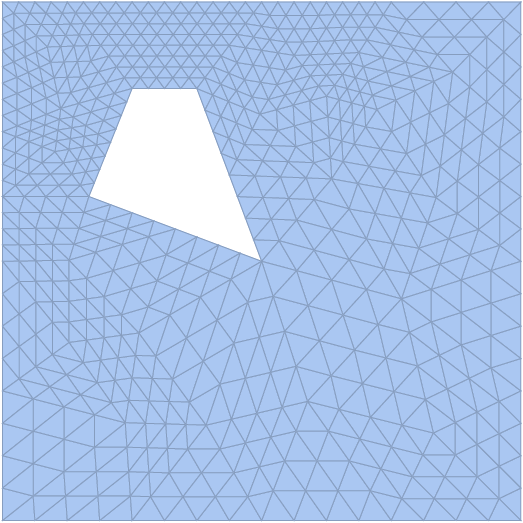}
      \label{polyann1}
    \end{center}
  \end{minipage}
  \caption{Two triangulated polygonal annuli}
 \hfill 
   \begin{minipage}[t]{.45\textwidth}
    \begin{center}  
      \includegraphics[width=2.25in]{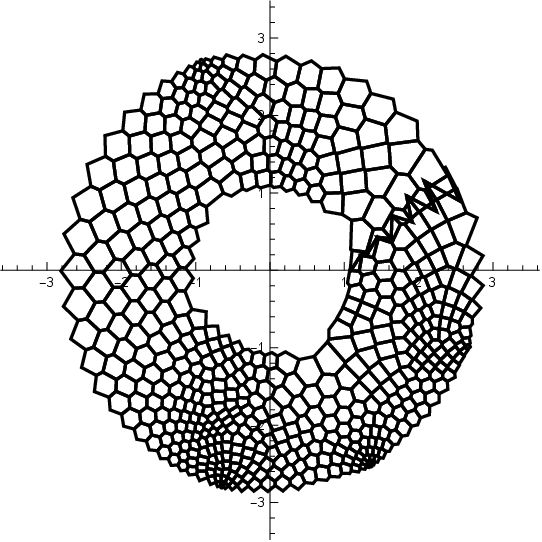}
      \label{voro1}
    \end{center}
  \end{minipage}
   \hfill 
   \begin{minipage}[t]{.45\textwidth}
   \begin{center}  
      \includegraphics[width=2.3in]{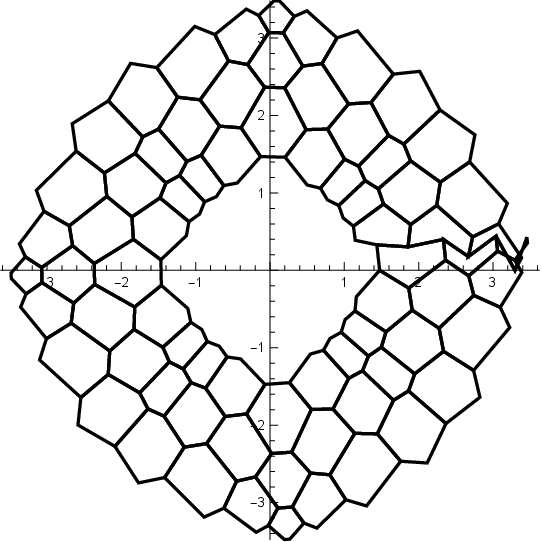}
      \label{voro2}
    \end{center}
  \end{minipage}
   \caption{The corresponding images of the Voronoi cells}
 \hfill 
 \vspace{.15In}
 \begin{minipage}[t]{.45\textwidth}
    \begin{center}  
      \includegraphics[width=2.4in]{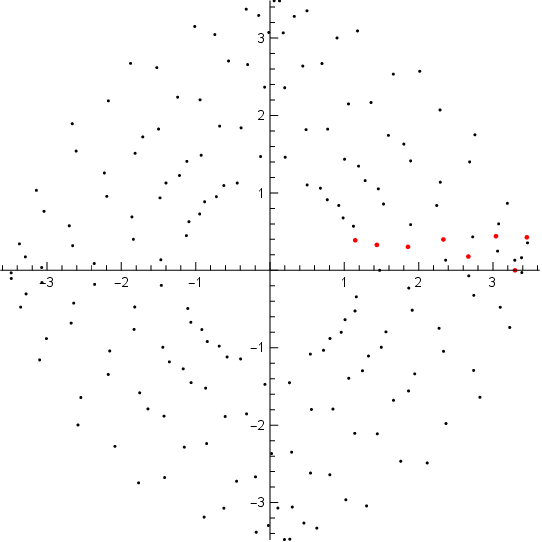}
      \label{image1}
    \end{center}
  \end{minipage}
  \hfill
  \begin{minipage}[t]{.45\textwidth}
    \begin{center}  
      \includegraphics[width=2.3in]{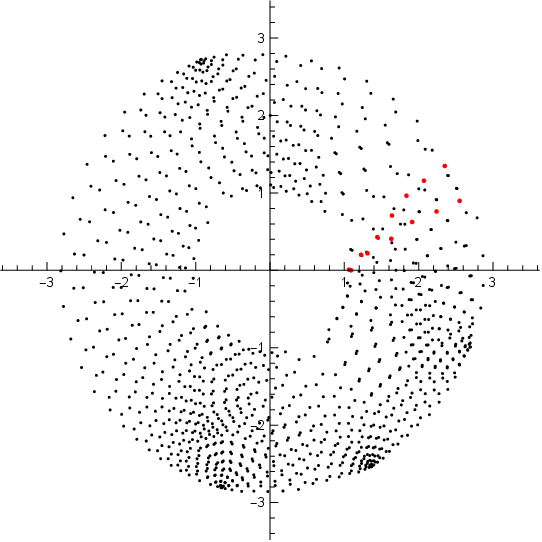}
      \label{image2}
    \end{center}
  \end{minipage}  
   \caption{Almost round annuli as the images of the polygonal annuli}
  \end{figure}


\section{The simply connected case}
\label{se:disk}
In this section, we will affirm Stephenson's Conjecture (Conjecture~\ref{con:step}) which was originally stated for the case of a bounded, simply connected, planar domain. Our point of departure is Theorem~\ref{th:annulus2} whose notation will be closely followed. 
The proof of this case entails on successively applying this theorem to an increasing sequence of annuli,  a known modification of Koebe's compactness theorem,  Riemann's  removable singularity theorem, a lemma concerning the monotonicity of periods, and a basic covering property of planar Riemann surfaces. 

\medskip

In the following, we will let 
\begin{equation}
\label{eq:invert}
\sigma(z)= \dfrac{1}{z},
\end{equation}
be the standard inversion of ${\mathbb C}$; it is well known that $\sigma$ is conformal. We can now turn to 

\medskip

\begin{Thm}
\label{th:disk}
Let $\Omega$ be a simply connected domain, embedded in $\mathbb{C}$ and  bounded by a closed, continuous curve $\Gamma$; let 
$p_0\in \Omega$ be a fixed point. Let $\{\Omega_n\}\subset \Omega$ be a nested sequence of disjoint, polygonal, Jordan disks with polygonal boundaries $\{\Theta_n\}$  such that the disks converge to $p_0$, that is,
\begin{equation}
 \Omega_1\supset \Omega_2\ldots \supset \Omega_k\ldots , 
 \end{equation}
 \begin{equation}
  \text{mesh}(\Omega_n)\rightarrow 0 \text{ as } n\rightarrow \infty,  \text{ and}
 \end{equation}
 \begin{equation}
 p_0=\cap_n \Omega_n.
 \end{equation}

For each $n$,  let 
${\mathcal A}_n = {\mathcal A}_n(\Omega,\Theta_n)$ be the  polygonal annulus defined by $\Omega \setminus \Omega_n$  with 
$\partial{\mathcal A}_n=\Gamma\cup\Theta_n$, endowed with  a sequence of quasi-uniform triangulations 
$\{{\mathcal T}_{m, {\mathcal A}_n  }\}_{m=1}^{\infty}$, such that for all $m=m( {\mathcal A}_n)$
large enough, ${\mathcal T}_{m, {\mathcal A}_n  }$
satisfies the hypotheses of Theorem~\ref{th:annulus2}.  Let 
\begin{equation}
\label{eq:maponeachannulus}
\Phi_{n}=\Phi_{n}({\mathcal A}_n):{\mathcal A}_n     \rightarrow  {\mathcal E}_{n}
\end{equation}
 be the sequence of conformal homeomorphisms constructed according to Equation~(\ref{eq;defofmaps}) 
onto the interior of concentric Euclidean annuli  ${\mathcal E}_{n}$, whose inner and outer radii are given by, respectively 
\begin{equation}
\label{eq:dimann1}
\{R_1,R_{2,n  }\}=  
\{1,  \exp\big( \frac{2\pi}{\mbox{\rm period}(u_{n}^{*})}\,     \big)\}, 
\end{equation}
where $u_{n}^{*}$ is the (smooth) harmonic conjugate of $u_n$,  the solution of the boundary value problem (\ref{de:strong}) defined on ${\mathcal A}_n$.

\smallskip
Then, a normalized subsequence of $\{\sigma\circ\Phi_n\}$ converges uniformly on compact subsets of $\Omega\setminus p_0$ to a holomorphic homeomorphism  $\Psi$  from 
$\Omega\setminus p_0$ onto $\mathbb{D}\setminus 0$.  Furthermore,  $\Psi$  can be extended 
to be holomorphic over $\Omega$.
\end{Thm}

\noindent{\it Proof.}
Following the rationale preceding Corollary~\ref{th:anncont}, we may assume that  $\Gamma$ is polygonal.
By construction, the $\{{\mathcal A}_n\}$ is a strictly increasing sequence, that is,
\begin{equation}
\label{eq:annuliunderinversion}
{\mathcal A_1} \subsetneq {\mathcal A_2}\subsetneq \ldots \subsetneq {\mathcal A_k}\ldots
\end{equation}
which all share  $\Gamma=\partial \Omega$ as their outer boundary component, and 
with $\Omega\setminus \{p_0\}$ being their union. The following lemma is needed in order to understand a monotonicity property of the sequence $\{ A_n\}$.


\smallskip

\begin{Lem}
\label{le:periodsdecrease}
The sequence $\{\text{period}(u_n^*)\}$ is strictly decreasing.
\end{Lem}

\dem
By Green's theorem, for all $n>1$ we have that, 
\begin{equation}
\label{eq:greesmooth}
\int_{{\mathcal A}_n}|\nabla u_n|^2 dx+ \int_{{\mathcal A}_n}\Delta u_n u_n  dx= \int_{\partial {\mathcal A}_n}\frac{\partial u_n}{\partial n} ds.
\end{equation}
However, by the definition of $\text{period}(u_n)$, and since $u_n$ is the solution of the boundary value problem (\ref{de:strong}) defined on ${\mathcal A}_n$, for all $n>1$, we have that 
\begin{equation}
\label{eq:greesmooth1}
\int_{{\mathcal A}_n}|\nabla u_n|^2 dx=\text{period}(u_n).
\end{equation}

It is clear that for all $n>1$, $u_n$ can be extended to be zero on ${\mathcal A}_{n+1}\setminus {\mathcal A}_{n+1}$ to a piecewise smooth function on ${\mathcal A}_{n+1}$ having the same boundary values as those of $u_{n+1}$. The assertion of the lemma now follows by the well-known characterization of $u_{n+1}$ as the unique minimizer of the Dirichlet integral over ${\mathcal A}_{n+1}$.

\eop{Lemma~\ref{le:periodsdecrease}}

\smallskip

It follows from Equation (\ref{eq:invert}), Equation (\ref{eq:dimann1}) and  the Lemma,  that the sequence $\{A_n=\sigma({\mathcal E}_n)\}$ consists of planar, concentric, Euclidean annuli,  such that  the inner and outer radii of each $A_n$ are given by 
\begin{equation}
\label{eq:dimann11}
\{r_1,r_{2,n  }\}=  
\{1, 1/ \exp\big( \frac{2\pi}{\mbox{\rm period}(u_{n}^{*})}\,     \big)\}, 
\end{equation}
  respectively; where  the sequence $\{ r_{2,n  }\}$ is strictly decreasing. Note that all the $A_n$'s 
 share ${\mathbb S}^{1} =\partial {\mathbb D}$ as their outer boundary component,  
\begin{equation}
\label{eq:imageannuli}
A_1\subsetneq A_2\subsetneq\ldots\subsetneq A_k\ldots,
\end{equation}
and the sequence $\{ A_n\}$ exhausts  ${\mathbb D}\setminus 0$.

\begin{figure}[htbp]
\begin{center}
 \scalebox{.50}{ \input{Riemann.pstex_t}}
 \caption{The evolution of $\Omega$.}
\label{figure:annulii}
\end{center}
\end{figure}

\smallskip
Pick $z_0\in A_1$, a local complex parameter at $z_0$, and a fixed $\xi_0\in {\mathbb C}$.
For all $n>1$, we now apply a normalization by post composing $\sigma \circ \Phi_n$ with a conformal embedding $f_n: A_n\rightarrow {\mathbb C}$ so that the composed maps 
\begin{equation}
\label{eq:normalizeedmaps}
\Xi_n = f_n\circ \sigma\circ \Phi_n : {\mathcal A}_n\rightarrow {\mathbb C}
\end{equation}
satisfy 
\begin{equation}
\label{eq:normalize}
\Xi_n (z_0)=\xi_0 \text{ and } \Xi_n '(z_0)=1.
\end{equation}

\smallskip
Note that the image of each $A_n$ is still a concentric Euclidean annulus, yet the sequence $\{\Xi_n({\mathcal A}_n)\}$ is not (generically) concentric. Nevertheless, 
it follows from a modification of Koebe's compactness theorem (see for instance \cite[Proposition 7.15]{ConII}) and a Cantor diagonalization process, that a subsequence of the $\{\Xi_n\}$ converges uniformly on compact subsets of $\Omega\setminus p_0$, to a conformal, univalent mapping  
\begin{equation}
\label{eq:the limiting map}
\Xi:\Omega\setminus p_0  \rightarrow   {\mathbb C},
\end{equation}
which is obviously not constant.
 It is also evident that $\Xi$ is bounded, and therefore, by  Riemann's removable singularity mapping theorem, can be extended to a conformal, univalent, embedding from $\Omega$. Hence, the extended map must be equal to the Riemann mapping with the same normalization. This ends the proof of the Theorem.
 
\eop{Theorem~\ref{th:disk}}

\medskip

Our current research is addressing the following themes.

\medskip

\paragraph{\bf 1. Disk packing and quasi-uniform triangulations.}
It is well known (see for instance \cite[Section 5]{Dub})  that a sequence of disk packings satisfying some minor conditions, induce  (as explained in \ref{eq:regularconstant}) a sequence of quasi-uniform triangulations, that will in addition satisfy assumptions (V0) and (V1). However,  assumption (V2) (see page 14) will not always be satisfied; it will be satisfied (for instance) for sub-packings of scaled copies of the infinite hexagonal disk packing (which were the subject of Thurston's original conjecture). Recall that assumption (V2) was used in the proof of Theorem~\ref{th:annulus2} \emph{only} in the part addressing the convergence of the $\bar g_n^{\ast}$. 
\bigskip

\paragraph{\bf  2. The case of higher connectivity.}
As mentioned in the introduction, Stephenson's original conjecture can be formulated for any
finitely connected, Jordan  domain. However, several issues need to be addressed before an appropriate statement can be made. For instance, the existence of singular points and level curves for smooth harmonic functions solving a Dirichlet problem (analogous to the one in Theorem~\ref{th:annulus2}) on such domains needs to be addressed. 

\bigskip
\paragraph{\bf  3.  Effective computational  approach through parallel processing.}
Polygon approximation of the boundary curves in the 2D or the 3D shape is essential for the computational aspects of this work and its successors. It will be used to smooth out any irregularities which may be present in the planar curves due to various effects, and  to achieve data reduction. In order to turn the computational problem to a parallel processing scheme, it is natural to utilize techniques from \cite{Ara, WanAra} and a decomposition of the given domain to subdomains as proposed in \cite{Her1}. Finally, we will compare our computational results with those obtained in \cite{AraOli} for the cases of scaling and rotation which are two useful transformations in the field of computer vision and imaging.



\newpage

\appendix

\section{preparatory facts}
\label{se:notation}

\subsection{The finite element method}
\label{se:notation1}

In this section, we will assume that  $\Omega$ is a fixed, bounded, $m$-connected with $m\geq 2$, polygonal domain in $\mathbb{R}^2$. 
The finite element method (FEM) is a powerful discretization scheme, aimed at constructing and presenting approximations of solutions of partial differential equations in the form of algebraic set of equations; the unknowns are 
the coefficients of a linear combination of the basis elements of a linear space comprising of simple functions: piecewise linear polynomials. 

\smallskip

The phrase 
``finite element" reflects on the process of approximation. 
 The domain in which the boundary value problem is define is divided into a collection of subdomains, where each subdomain is presented by a set of element equations derived from  the original problem. One then systematically combines all of the sets of element equations into a global system of equations for the final calculation.

\smallskip

We now turn to a specific boundary value problem that we will consider, the homogeneous Dirichlet  equation: Given $f \neq 0\in L^{2}(\Omega)$, find $\tilde u\in C^2(\Omega)\cap C(\bar \Omega)$ satisfying 
\begin{equation}
\label{eq:diffusion}
 \triangle \tilde u = f\  \mbox{\rm in}\  \Omega,\  \text{and } \tilde u=0 \text{ on } \partial\Omega.
\end{equation}

Such a solution is called a {\it strong  solution}.



\smallskip

A common theme in finite element method is to triangulate the domain in a geometric 
convenient way.


\begin{Def}
\label{de:triangulation}
A triangulation ${\mathcal T}$ of $\Omega$ is a set of (closed) triangles $T_i$, $i=1,\ldots,n$ such that the following hold
\begin{equation}
\bar\Omega=\bigcup_{i=1}^{n}T_i, \ \mbox{\rm and}\ T_i\cap T_j=\emptyset, \ \mbox{\rm a vertex or one common edge},\ 
\mbox{\rm for all}\ i\neq j.
\end{equation}
\end{Def}

The following  quantity is associated with a fixed triangulation. 
\begin{Def}
\label{de:mesh}
Let ${\mathcal T}$ be a triangulation on $\Omega$, the mesh size of ${\mathcal T}$ is equal to 
\begin{equation}
\sup_{T\in {\mathcal T}} d(T),
\end{equation}
where $d(T)$ denotes the diameter of $T$ (i.e., the length of its largest edge). Henceforth, ${\mathcal T}_{\rho}$ will denote a triangulation of $\Omega$ of mesh size that is equal to $\rho$.
\end{Def}

In order to apply the machinery of numerical approximation of elliptic boundary value problems, one needs to avoid situations where triangles, in any triangulation
 ${\mathcal T}_{\rho}$ of the domain, become flat as $\rho\rightarrow 0$. To this end, we let $\sigma(T)$ denote the diameter of the largest circle that can be inscribed in a triangle $T$. We now define the geometric property of the special class of triangulations that will be used throughout this paper.

\begin{Def}
\label{de:regular}
A family of triangulations $\{ {\mathcal T}_\rho\}$ of $\Omega$ is called  $\tau$-{\it quasi-uniform} (or  $\tau$-{\it quasi-regular})  when $\rho\rightarrow 0$, if there exists a positive constant $\tau$ such that
\begin{equation}
\label{eq:regularconstant}
\frac{d(T)}{\sigma(T)}\leq \tau \   \mbox{\rm for all} \ T\in {\mathcal T}_\rho,\  \mbox{\rm and for all}\ \rho \text{ small enough}.
\end{equation}
\end{Def}

\smallskip


Given a triangulation we will now associate to it a vector space of functions.

\begin{Def}[]
\label{de:elementspace}
\begin{equation}
\label{eq:finitespace}
\mathbb{V}_{0,{\mathcal T}}= \{ \phi :\Omega\rightarrow \mathbb{R}\  | \phi \in C(\bar\Omega),\ v|{T}\in {\mathcal P}_{1}(T) \ \mbox{\rm for all}\ T\in{\mathcal T}\ \text{and}\   \phi=0\ \mbox{on } \partial\Omega\},
\end{equation}
where ${\mathcal P}_{1}(T)$ denotes the space of linear polynomials in two variables over $T$. 
\end{Def}
Let $V^{0}({\mathcal T})$ denote  the set of vertices in ${\mathcal T}^{(0)}$ which are in the interior of $\Omega$, and set 
\begin{equation}
\label{eq:nuofvertices}
M_1= |V^{0}({\mathcal T})|, M_2=|{\mathcal T}^{(0)}  \cap\partial\Omega| \text{ and }  M=M_1+M_2\  (=|{\mathcal T}^{(0)}|).
\end{equation}
  It is well  known that   $\mathbb{V}_{0,{\mathcal T}}$ is a finite dimensional vector space which is spanned by $\{\phi_i\}$ - {\it the nodal basis};  where by definition 
\begin{equation}
\label{eq:conditions}
\phi_i(x_j)=\delta_{i,j}, \text{for all } x_j\in V^{0}({\mathcal T}) ,i=1,\ldots,M_1. 
\end{equation}

One important feature of $\mathbb{V}_{0,{\mathcal T}}$ is that it is 
a linear subspace of a certain Sobolev space.   Let us first recall 
\begin{Def}
\label{de:SobSpace}
The Sobolev space $H^{1,2}(\Omega)$ is the subset of $L^{2}(\Omega)$ defined by
\begin{equation}
\label{eq:SobSpace}
H^{1,2}(\Omega)= \{ v\in L^{2}(\Omega) \ |\  \partial_x v, \partial_y v \in  L^{2}(\Omega)\},  
\end{equation}
where $\partial_x v, \partial_y v$ denote the distributional derivatives of $v$ in the $x$ and the $y$ directions, respectively. The integration is with respect  to the standard Lebesgue measure in the plane which will be denoted by $dx$.

\end{Def}
For $u,v\in H^{1,2}(\Omega)$, one defines the scalar product and an associated norm, respectively, by  
\begin{equation}
\label{eq:scalarproduct}
(u,v)_{1,2}= \int_{\Omega} (uv +\nabla u \cdot \nabla v),\  |u|_{1,2}^2=(u,u)_{1,2}
\end{equation}
where $\nabla v= (\partial_x v,\partial_y v)$, and the scalar product is the Euclidean one  in $\mathbb{R}^{2}$.
It is well known that  $H^{1,2}(\Omega)$  equipped with this scalar product  is a Hilbert space. Finally, let $H^{1,2}_{0}(\Omega)$ be defined as the closure of $C^{\infty}_{0}(\Omega)$ in $H^{1,2}(\Omega)$ with respect to this norm.   Equipped with this scalar product $H^{1,2}_{0}(\Omega)$ is a Hilbert space as well. It is a useful fact that  $\mathbb{V}_{0,\mathcal T}$  is a linear subspace of $H^{1,2}_{0}(\Omega)$.  
 
\medskip

The first step in finite element method amounts to finding the {\it weak solution} of  the boundary value problem  (\ref{eq:diffusion}).  That is, finding  $u\in H^{1,2}_{0}(\Omega)$ such that 
 
 \begin{equation}
 \label{eq:intalongvolumes}
 \int_{\Omega} \nabla u\cdot \nabla v\, dx =\int_{\Omega} f v\,dx, \ \text{for all }\ v\in H^{1,2}_{0}(\Omega).
 \end{equation}

Next, one replaces the space $H^{1,2}_{0}(\Omega)$ by a {\it sequence} of linear subspaces that exhaust it, i.e., one lets $\rho$ converge to zero and search for a solution of  the above equation in
$\mathbb{V}_{0,{\mathcal T}_\rho}$: Finding $u_\rho \in \mathbb{V}_{0,{\mathcal T}_\rho}$ such that 

\begin{equation}
 \label{eq:intalongvolumes1}
 \int_{\Omega} \nabla u_\rho\cdot \nabla v_\rho \, dx =\int_{\Omega} f v_\rho\  dx, \ \text{for all }\ v_\rho\in \mathbb{V}_{0,{\mathcal T}_\rho}.
\end{equation}

In section~\ref{se:ENFT}, we  recall a foundational result concerning the  convergence of the piecewise linear polynomials $u_\rho$  as $\rho$ converges to zero.


 \subsection{Stephenson's conductance constants - a finite element method perspective}
 \label{se:ENDF}
 
In this section, we will work with an explicit form of the solution $u_\rho$ of 
 (\ref{eq:intalongvolumes1}).  We will recall the known fact that $u_\rho$ (for each $\rho$) can be derived form the solution of a system of finitely many linear equations. Following this, 
we will define a quantity,  the {\it discrete flux} of $u_\rho$, which will be used  in this paper to approximate the analytical flux of $u$, the solution of (\ref{eq:diffusion}). To this end, note that the interpolation conditions in (\ref{eq:conditions}) allows us to write in a unique way
\begin{equation}
\label{eq:coefficients}
u_\rho(x) = \sum_{i=1}^{M_1}u_{\rho}(x_i)\phi_i(x),\ \text{ for } x\in \Omega,
\end{equation}
where the unknowns are $ u_{\rho}(x_i)$ for all $i=1,\dots, M_1$. 
Therefore, if we define  $u_{\rho}(i)=u_{\rho}(x_i)$ for all $i=1,\ldots, M_1$, then our vector of unknowns for $u_\rho$ is given by  
${\bf \xi}_\rho=(u_\rho(1), \ldots, u_\rho(M_1))$. Hence, (\ref{eq:intalongvolumes1}) can be written as the   following matrix equation
\begin{equation}
\label{eq:matrix}
{\bf A}_{\rho}{\bf\xi}_{\rho} = {\bf q}_{\rho}, 
\end{equation} 
where we have for all $i,j =1,\ldots, M_1$ 
\begin{equation}
\label{eq:matrixpresentation}
\begin{array}{ccl}
  (\bf A_\rho)_{i,j}&   = &\int_{\Omega}\nabla\phi_j \cdot \nabla\phi_i   \\
  &   &  \\
  (\bf q_\rho)_{i}& =  &\int_{\Omega} f \phi_i     
\end{array}
\end{equation}

\medskip

We keep the notation as in the discussion preceding Figure~\ref{figure:ControlVolume}.
The following lemma and its corollary allow us to turn each ${\mathcal T}_\rho$ into a finite network and provide  a relation between Stephenson's constants and the Finite Element Method.                  

\begin{Lem}[\mbox{\cite[Lemma 6.8]{AnKn}}]
\label{le:asfinitedifference}
 Let ${\mathcal T}$ be any fixed triangulation of\ 
$\Omega$, and consider its corresponding Voronoi diagram. Then, for an 
arbitrary triangle $T\in  {\mathcal T}_\rho$ with vertices $x_i,x_j (i\neq j)$, the following relation holds
\begin{equation}
\label{eq:overatriangle}
\int_{T} \nabla \phi_j \cdot \nabla \phi_i\  dx= - \frac{m_{(i,j)}^{T}}{d_{ij}},
\end{equation} 
where $m_{(i,j)}^{T}$ is the length of the segment of $\Gamma_{ij}$ which intersects T.
\end{Lem}

A computation then shows that
\begin{Cor}[\mbox{\cite[Corollary 6.9]{AnKn}}]   
\label{co:innerproductformulation}
Under the assumptions of Lemma ~\ref{le:asfinitedifference}, we have 
\begin{equation}
\label{eq:finitedifference}
\int_{\Omega} \nabla u_\rho \cdot \nabla \phi_i \ dx = \sum_{x_j\sim x_i, j\neq i}\frac{m_{(i,j)}}{d_{ij}} \big(u_{\rho}(x_i) -u_{\rho}(x_j) \big).
\end{equation}
\end{Cor}

Hence, by letting the index $i$ range over the indices of the interior vertices (i.e, those that are in $V_{\rho}^{0}({\mathcal T}_{\rho})$),  (\ref{eq:matrixpresentation}) turns into the following  {\it system of linear equations} 

\begin{equation}
\label{eq:linearsystem}
\sum_{x_j\sim x_i,j\neq i}\frac{m_{(i,j)}}{d_{ij}} \big(u_{\rho}(x_i) -u_{\rho}(x_j) \big) = \int_{\Omega} f \phi_i\ dx, \text{ for all } i=1,\ldots, M.
\end{equation}

\smallskip

\begin{Rem}
Note that when $f\equiv 0$, $u_{\rho}$ is a {\it discrete} harmonic function on ${\mathcal T}_{\rho}^{(0)}$ with the conductance constant $\frac{m_{(i,j)}}{d_{ij}}$ for the edge joining $x_i$ to $x_j$.
\end{Rem}




\end{document}